\documentclass[reqno]{amsart}

\usepackage{epsfig, enumerate}
\usepackage{graphicx}
\usepackage[usenames,dvipsnames,svgnames,table]{xcolor}
\usepackage{cases}

\usepackage[charter]{mathdesign}

\usepackage{tikz}
\usetikzlibrary{backgrounds}
\usetikzlibrary{patterns,fadings}
\usetikzlibrary{arrows,decorations.pathmorphing}
\usetikzlibrary{calc}
\definecolor{light-gray}{gray}{0.95}
\usepackage{float}
\usepackage[colorlinks=true,linkcolor=blue,citecolor=magenta]{hyperref}

\newcommand{\pfrac}[2]{\genfrac{}{}{}{1}{#1}{#2}}

\newtheorem{theorem}{Theorem}[section]
\newtheorem{lemma}[theorem]{Lemma}
\newtheorem{proposition}[theorem]{Proposition}
\newtheorem{corollary}[theorem]{Corollary}
\newtheorem{remark}[theorem]{Remark}
\newtheorem{definition}{Definition}[section]
\newtheorem{assumption}{Assumption}[section]

\numberwithin{equation}{section}

\newcommand{\mc}[1]{{\mathcal #1}}

\newcommand{\bb}[1]{{\mathbb #1}}

\newcommand{\<}{\langle}
\renewcommand{\>}{\rangle}

\renewcommand{\epsilon}{\varepsilon}

\newcommand{\p}{\partial}
\newcommand{\eps}{\varepsilon}
\def\centerarc[#1](#2)(#3:#4:#5){\draw[#1] ($(#2)+({#5*cos(#3)},{#5*sin(#3)})$) arc (#3:#4:#5);}

\newcommand{\XX}{\mathbb{X}}

\newcommand{\Y}{\mathcal{Y}}

\newcommand{\at}[2]{\genfrac{}{}{0pt}{}{#1}{#2}}

\newcommand{\Ind}[1]{\textbf{1}{_{\lbrace{#1}\rbrace}}}           

\newcommand{\A}{\mathcal{A}}

\title[Fluctuations of  a slowed boundary symmetric exclusion process]{Non-equilibrium and stationary fluctuations of a slowed boundary symmetric exclusion}



\author{Tertuliano Franco}
\address{UFBA\\
 Instituto de Matem\'atica, Campus de Ondina, Av. Adhemar de Barros, S/N. CEP 40170-110\\
Salvador, Brasil}
\email{tertu@ufba.br}
\thanks{}

\author{Patr\'{\i}cia Gon\c{c}alves}
\address{\noindent Center for Mathematical Analysis,  Geometry and Dynamical Systems \\
Instituto Superior T\'ecnico, Universidade de Lisboa \\
Av. Rovisco Pais, 1049-001 Lisboa, Portugal}
\email{patricia.goncalves@math.tecnico.ulisboa.pt}

\author{Adriana Neumann}
\address{UFRGS, Instituto de Matem\'atica, Campus do Vale, Av. Bento Gon\c calves, 9500. CEP 91509-900, Porto Alegre, Brasil}
\email{aneumann@mat.ufrgs.br}
\thanks{}

\begin{document}

\begin{abstract}
We consider a one-dimensional symmetric   simple exclusion process in contact with slowed reservoirs: at the left (resp. right) boundary, particles are either created or removed  at rates given by  $\alpha/n$ or $(1-\alpha)/n$ (resp. $\beta/n$ or $(1-\beta)/n$) where $\alpha, \beta>0$ and $n$ is a scaling   parameter. We obtain the non-equilibrium fluctuations  and consequently the non-equilibrium stationary fluctuations.
\end{abstract}

\maketitle


\section{Introduction}\label{s1}

One of the most intriguing problems in the field of interacting particle systems is the rigorous mathematical derivation of the non-equilibrium fluctuations of a system around its hydrodynamic limit. The main difficulty one faces when trying to show that result is the fact that the systems exhibit long range space-time correlations and for that reason the non-equilibrium fluctuations   have only  been derived for very few models, see, for example, \cite{Yau1992,jaralandim2008,Ravi1992} and references therein. 
Moreover, the study of non-equilibrium steady states has attracted a lot of attention over the last twenty years and up to now the microscopic description of these states is still {incipient},  see, for example,  the review~\cite{blytheevans}.

In \cite{lmo}, the non-equilibrium stationary fluctuations for the symmetric simple exclusion in contact with fixed reservoirs {were} derived as a simple consequence of  its non-equilibrium fluctuations. In this article we examine the dynamical non-equilibrium fluctuations of the symmetric simple exclusion process in contact with \emph{slowed} reservoirs. In this model the exclusion dynamics is superposed with a Glauber dynamics at each end point of a one-dimensional lattice with $n-1$ points. According to this dynamics, particles perform continuous time symmetric random walks in the discrete lattice $\{1,\ldots, n-1\}$ to which we call bulk, in such a way that two particles cannot occupy the same site at a given time, the {so-called} exclusion rule. Moreover, at the end points of the bulk we add two extra sites, namely $0$ and  $n$, corresponding to two different reservoirs  where  particles can be created or annihilated at a certain rate, which is slowed with respect to the jump rate in the bulk.

Our main interest is the  derivation of the non-equilibrium fluctuations and the non-equilibrium stationary fluctuations for this model.  We chose a regime in which the Glauber dynamics is slowed enough so that the hydrodynamic behavior of the system is macroscopically different from  the case in which the Glauber dynamics is not slowed, as in \cite{lmo}, for instance. More precisely, in \cite{lmo} the Glauber dynamics is defined in such a way that particles can get in and out of the system at rate $\alpha$ and $\beta$, respectively. In our model, these rates are slowed by a parameter $n$. As a consequence of  having slowed reservoirs,  the hydrodynamical profile in our model is different from the one of \cite{lmo}, the latter being a solution of the heat equation with Dirichlet boundary conditions  in which the solution is fixed at the boundaries by  $ \rho(t,0) = \alpha$ and 
$\rho(t,1) = \beta$.  In the model considered here, it has been proved in \cite{bmns} that the hydrodynamical profile  is a solution of the heat equation with a type of Robin boundary conditions in which the value of the profile at the boundaries is not fixed, but instead it fixes the values of its space derivative, namely: $\p_u \rho(t,0) = \rho(t,0)-\alpha$ and $\p_u \rho(t,1) = \beta-\rho(t,1)$, see \eqref{hydroeq}. These boundary conditions reflect the fact that the mass transfer, given by $\p_u \rho(t,\cdot)$, at the boundaries is proportional to  the difference of concentration. Contrarily to what happens in the model of \cite{lmo} which fixes the density at the reservoirs, in our case we do not have $\rho(t,0)=\alpha$, so that the term $\rho(t,0)-\alpha$ represents the difference of concentration between the bulk and the boundary.   We also note that in \cite{bmns} it  has been analyzed the hydrodynamic limit for  a generalization of our model. There, the rates at the reservoirs are slowed with respect to the rate in the bulk by a factor $n^\theta$, where $\theta>0$, and our model corresponds to the choice  $\theta=1$. We note that, as proved in \cite{bmns}, for $\theta<1$ (resp. $\theta>1$) the hydrodynamical profile is the unique weak solution of the heat equation with Dirichlet (resp. Neumann) boundary conditions.

We would also like to refer other important articles on this subject as, for example,   \cite{mariaeulalia1,mariaeulalia3,mariaeulalia2}, where the authors consider models with  slowed boundaries but  one boundary acts only for the creation of particles and the other boundary acts only on the annihilation of particles.  As a consequence, the density of particles in the reservoirs remains the same, and the hydrodynamical profile  in such case is a solution of the heat equation with Dirichlet boundary conditions.

We observe that when $\alpha=\beta=\rho$, the reservoirs do not induce any current in the system contrarily to what happens if, for example, $\alpha<\beta$, since in this case particles can get in the system more easily from the right boundary, and there is a current of particles, due to the reservoirs, from the right reservoir  to the left reservoir.  In the case $\alpha=\beta=\rho$, the Bernoulli product measures given by $\nu_\rho\{\eta:\eta(x)=1\}=\rho$ are invariant and due to the absence of an external current, they are called equilibrium measures.  However, in the non-equilibrium scenario, that is when $\alpha \neq \beta$, this fact is no longer true. Nevertheless, there exists a unique stationary measure that we denote by $\mu_{ss}$.  Since $\alpha\neq\beta$, the reservoirs induce a current of particles in the system and for that reason $\mu_{ss}$ is a non-equilibrium stationary measure. This measure has been partially characterized in, for example, \cite{d}
 and it has been proved in \cite[Theorem 2.2]{bmns} that it is associated to a profile $\bar\rho(\cdot)$ which is stationary with respect to the hydrodynamic equation, so that $\bar\rho(\cdot)$ is linear and $\bar\rho(0)=\tfrac{2\alpha+\beta}{3}$ and $\bar\rho(1)=\tfrac{\alpha+2\beta}{3}$.  We emphasize here that, as one can see from the previous properties on the stationary profile, in our model the density at the reservoirs is not fixed as being $\alpha$ at $u=0$ and $\beta$ at $u=1$.

 To analyze the non-equilibrium fluctuations we consider a space of smooth test functions $f$ satisfying the boundary conditions of the homogeneous hydrodynamic equation, that is, the hydrodynamic equation with $\alpha=\beta=0$.  Our setting for initial states is quite general and can be described as follows. We consider  initial measures $\mu_n$ associated to a measurable profile $\rho_0:[0,1]\to\bb [0,1]$ in the sense of \eqref{eq3}. Moreover, denoting for $x\in\{1,\cdots,n-1\}$, $\rho_0^n(x)=\bb E_{\mu}[\eta(x)]$, we ask $\rho^n_0(\cdot)$ to be close to the given $\rho_0(\cdot)$ as stated in Assumption \ref{assumption1} and we also ask that   the corresponding  space correlations to vanish as $n\to\infty$, as stated in Assumption \ref{assumption2}. In this case we show that the sequence of density fluctuation fields  is tight and we characterize its limiting points so that, for a fixed time $t$, the solution is given by  the sum of a Gaussian random variable and the initial condition, see the relation \eqref{characterization}. Besides that, if on top of  the aforementioned assumptions we ask that at the initial time the sequence of density fields converges to a mean-zero Gaussian process, then the convergence takes place and the limiting process is an Ornstein-Uhlenbeck process solution of \eqref{O.U.}. We also note that from our results we can obtain the non-equilibrium fluctuations starting from a local Gibbs state. More precisely, if we fix a profile $\gamma:[0,1]\to[0,1]$ and  consider $\mu_n$ as the Bernoulli product measure such that $\mu_n\{\eta:\eta(x)=1\}=\gamma(\tfrac{x}{n})$, then the result also holds,  leading to an Ornstein-Uhlenbeck process in the limit. 
 
 As a consequence of the  previous results we can derive the non-equilibrium stationary fluctuations. For that purpose we just have to check that the imposed conditions on the initial states are satisfied by the non-equilibrium stationary state and to recover the corresponding covariance we perform a careful analysis of the time limit of the  covariance obtained in the general  non-equilibrium scenario. 

To prove the non-equilibrium fluctuations, since we consider the system starting from general initial measures, which can develop long range correlations,  we need  a sharp bound on the space correlations in order  to make our method work. For that purpose we make a careful analysis of solutions of a bidimensional  discrete scheme which has non-trivial boundary conditions. 

As a future work we plan to derive our results for the models studied in \cite{bmns} for the case  $\theta\neq 1$. The main difficulty we will face is the derivation of sharp bounds on the space correlations of the system, and we will also need  to perform  a careful analysis of some additive functionals associated to the system.

Here follows an outline of this article. In Section \ref{s2} we present the model, we recall its hydrodynamic limits and we enunciate our results, namely: Theorem \ref{non_eq_flu}, where we state the non-equilibrium fluctuations for general initial measures;  Theorem \ref{thm27}, where we state the non-equilibrium fluctuations when the limit is an Ornstein-Uhlenbeck process for which the initial measures have to satisfy a Gaussian central limit theorem and, as a consequence of the previous results; Theorem \ref{flustat} where we state the non-equilibrium stationary fluctuations. In Section \ref{sec:sem} we present some  necessary results related to the hydrodynamic equation and its semigroup. In Sections \ref{s3}, \ref{s4} and \ref{s5} we prove, respectively, Theorems \ref{thm27}, \ref{non_eq_flu} and \ref{flustat}. Section \ref{s6} is devoted to tightness and Section \ref{s7} is devoted to space correlations estimates.  
\section{Statement of results}\label{s2}

\subsection{The model}

Given $n\geq{1}$ let  $\Sigma_n=\{1,\ldots,n-1\}$. The symmetric simple exclusion process with slow boundaries is a Markov process $\{\eta_t:\,t\geq{0}\}$ with state space $\Omega_n:=\{0,1\}^{\Sigma_n}$. We denote the configurations of the state space $\Omega_n$ by $\eta$, so that for $x\in\Sigma_n$,  $\eta(x)=0$ means that the site $x$ is vacant while $\eta(x)=1$ means that the site $x$ is occupied.  We characterize this Markov process  in terms of its infinitesimal generator $\mc L_{n}$ as follows. Let $\mc L_{n}=\mc L_{n,o}+\mc L_{n,b}$, where, for a given a function $f:\Omega_n\rightarrow \bb{R}$,  we have 
\begin{equation}\label{lnb}
\begin{split}
(\mc L_{n,o}f)(\eta)\;=\;
\sum_{x=1}^{n-2}\Big(f(\eta^{x,x+1})-f(\eta)\Big)\,, 
\end{split}
\end{equation}
\begin{equation}\label{lno}
(\mc L_{n,b}f)(\eta)\;=\;\frac{1}{n}
\sum_{x\in\{1,n-1\}} \Big[{r_x}(1-\eta(x))+(1-r_x)\eta(x)\Big]\Big(f(\sigma^{x} \eta)-f(\eta)\Big)\,,
\end{equation}
with $r_1=\alpha$ and  $r_{n-1}=\beta$. Above,  for $x\in\{1,\ldots, n-2\}$, the configuration $\eta^{x,x+1}$ is obtained from $\eta$ by exchanging the occupation variables $\eta(x)$ and $\eta(x+1)$, i.e.,
\begin{equation*}
(\eta^{x,x+1})(y)\;=\;\left\{\begin{array}{cl}
\eta(x+1)\,,& \mbox{if}\,\,\, y=x\,,\\
\eta(x)\,,& \mbox{if} \,\,\,y=x+1\,,\\
\eta(y)\,,& \mbox{otherwise,}
\end{array}
\right.
\end{equation*}
 and for $x\in\{1,n-1\}$ the configuration  $\sigma^x\eta$ is obtained from $\eta$ by flipping  the occupation  variable $\eta(x)$, i.e,
  \begin{equation*}
(\sigma^x\eta)(y)=\left\{\begin{array}{cl}
1-\eta(y)\,,& \mbox{if}\,\,\, y=x\,,\\
\eta(y)\,,& \mbox{otherwise.}
\end{array}
\right.
\end{equation*}
The dynamics of this model can be described in words in the following way. In the bulk, particles move accordingly to continuous time symmetric random walks under the additional exclusion rule: whenever a particle tries to jump to an occupied site, such jump is suppressed. Additionally, at the left boundary, particles can be created (resp. removed) at rate $\alpha/ n$ (resp. at rate $(1-\alpha)/n$) and at the right boundary, particles can be created (resp. removed) at rate $\beta/ n$ (resp. at rate $(1-\beta)/ n$).  See Figure~\ref{fig1} for an illustration.
\begin{figure}[!htb]
\centering
\begin{tikzpicture}
\centerarc[thick,<-](0.5,0.3)(10:170:0.45);
\centerarc[thick,->](0.5,-0.3)(-10:-170:0.45);
\centerarc[thick,->](2.5,-0.3)(-10:-170:0.45);
\centerarc[thick,<-](3.5,0.3)(10:170:0.45);
\centerarc[thick,<-](8.5,-0.3)(-10:-170:0.45);
\centerarc[thick,->](8.5,0.3)(10:170:0.45);
\draw (1,0) -- (8,0);

\shade[ball color=black](1,0) circle (0.25);
\shade[ball color=black](3,0) circle (0.25);
\shade[ball color=black](5,0) circle (0.25);
\shade[ball color=black](6,0) circle (0.25);
\shade[ball color=black](8,0) circle (0.25);

\filldraw[fill=white, draw=black]
(2,0) circle (.25)
(4,0) circle (.25)
(7,0) circle (.25);

\draw (1.3,-0.05) node[anchor=north] {\small  $\bf 1$}
(2.3,-0.05) node[anchor=north] {\small $\bf 2 $}
(8.5,-0.05) node[anchor=north] {\small $\bf n\!-\!1$};
\draw (0.5,0.8) node[anchor=south]{$\alpha/n$};
\draw (0.5,-0.8) node[anchor=north]{$(1-\alpha)/n$};
\draw (3.5,0.8) node[anchor=south]{$1$};
\draw (8.5,-0.8) node[anchor=north]{$(1-\beta)/n$};
\draw (8.5,0.8) node[anchor=south]{$\beta/n$};
\draw (2.5,-0.8) node[anchor=north]{$1$};
\end{tikzpicture}
\caption{Illustration of  jump rates. The leftmost and rightmost rates are the entrance/exiting rates.}\label{fig1}
\end{figure}
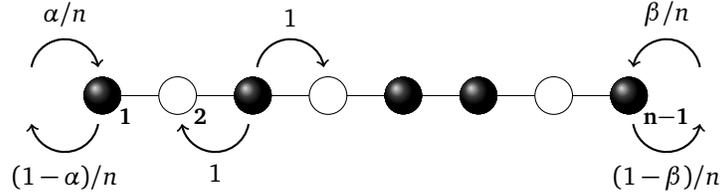
Note that when $\alpha=\beta=\rho$, for which there is no external current induced by the reservoirs, it is easy to check that  the Bernoulli product measures given by $\nu_\rho\{\eta:\eta(x)=1\}=\rho$ are invariant. However, when $\alpha \neq \beta$ this is no longer true.  Nevertheless, for $\alpha\neq \beta$, there is a unique stationary measure of the system, that we denote by  $\mu_{ss}$, which is no longer a product measure. For further properties on this measure we refer the reader to, for example, \cite{d}. In particular, it is shown in \cite[Theorem 2.2]{bmns} that this measure is associated to a profile $\bar\rho(\cdot)$ which is stationary with respect to the hydrodynamic equation, so that $\bar\rho(\cdot)$ is linear and $\bar\rho(0)=\alpha+\tfrac{\beta-\alpha}{3}$ and $\bar\rho(1)=\alpha+2\tfrac{\beta-\alpha}{3}$. We observe that in the case where the reservoirs are not slowed, as in \cite{lmo}, the  stationary profile associated to the hydrodynamic equation, which is the heat equation with Dirichlet boundary conditions,  is the  linear interpolation between $\alpha$ and $\beta$.

\subsection{Hydrodynamic limit}
Fix a measurable density profile $ \rho_0: [0,1] \rightarrow [0,1 ]$. For each  $n \in \bb N$, let $\mu_n$ be a probability measure on $\Omega_n$.  We say that the sequence $\{\mu_n\}_{n\in \bb N}$ is \textit{associated} to the profile $\rho_0(\cdot)$ if,  for any $ \delta >0 $ and any continuous function $ f: [0,1]\to\bb R $ the following limit holds:
\begin{equation}\label{eq3}
\lim_{n\to\infty}
\mu_{n} \Bigg[ \eta:\, \Big| \frac {1}{n} \sum_{x = 1}^{n\!-\!1} f(\pfrac{x}{n})\, \eta(x)
- \int f(u)\, \rho_0(u)\, du \Big| > \delta\Bigg] \;=\; 0\,.
\end{equation}
Fix $T>0$. Let  $\mc D([0,T],\Omega_n)$ be  the space of trajectories which are right continuous, with left limits and  taking values in $\Omega_n$.
Denote by $\bb P_{\mu_{n}} $ the probability on  $\mc D([0,T],\Omega_n)$ induced by the Markov process with generator $n^2\mc L_n$ and the initial measure $\mu_n$ and denote by $\bb E_{\mu_n}$ the expectation with respect to $\bb P_{\mu_n}$.
From \cite{bmns} we have the following result, known in the literature as \emph{hydrodynamic limit}.
\begin{theorem}[Hydrodynamic limit, \cite{bmns}]
\quad

 Suppose that the 
sequence $\{\mu_n\}_{n\in \bb N}$ is \textit{associated} to a profile $\rho_0(\cdot)$ in the sense of \eqref{eq3}. 
Then,  for each $ t \in [0,T] $, for any $ \delta >0 $ and any continuous function $ f:[0,1]\to\bb R $, 
\begin{equation*}
\lim_{ n \rightarrow +\infty }
\bb P_{\mu_{n}} \Bigg[\eta_{\cdot} : \Big\vert \frac{1}{n} \sum_{x =1}^{n-1}
f(\pfrac{x}{n})\, \eta_{tn^2}(x) - \int f(u)\, \rho(t,u)\, du\, \Big\vert
> \delta \Bigg] \;=\; 0\,,
\end{equation*}
 where  $\rho(t,\cdot)$ is the unique weak solution of the 
 heat equation with certain   Robin boundary conditions given by
\begin{equation}\label{hydroeq}
\begin{cases}
\p_t \rho(t,u)= \p_u^2 \rho(t,u)\,, & \textrm{ for } t>0\,,\, u\in (0,1)\,,\\
\p_u \rho(t,0) = \rho(t,0)-\alpha\,, & \textrm{ for } t>0\,,\\
\p_u \rho(t,1) = \beta-\rho(t,1)\,, & \textrm{ for } t>0\,,\\
\rho(0,u)=\rho_0(u)\,,& u\in [0,1]\,.
\end{cases}
\end{equation}
\end{theorem}

\subsection{Density fluctuations}
\subsubsection{The space of test functions}
By $f\in C^\infty([0,1])$ we mean that both $f:[0,1]\to \bb R$ as well as  all its  derivatives are continuous functions in $[0,1]$. Next, we define a subspace of $C^\infty([0,1])$ which is intrinsically associated to the limiting fluctuations, as we shall see later on.
\begin{definition}\label{def1} Let $\mathcal S$ denote the set of functions $f\in C^\infty([0,1])$ such that for any $k\in\mathbb{N}\cup \{0\}$ it holds that 
$ \p_u^{2k+1} f(0)= \p_u^{2k}f(0)$ and $\p_u^{2k+1} f(1)=- \p_u^{2k}f(1).
$
\end{definition}
Notice that  for $k=0$, the conditions above are nothing else than the boundary conditions that appear in the homogeneous version of \eqref{hydroeq}, i.e., imposing $\alpha=\beta=0$. For $k=1$, the conditions above are again these boundary conditions, but imposed for the Laplacian of $f$, and so on.

\begin{definition}\label{def2}
Let $T_t:\mc S\to\mc S$ be the semigroup associated to \eqref{hydroeq} with $\alpha=\beta=0$. That is, given $f\in \mc S$, by $T_tf$ we mean the solution of the homogeneous version of \eqref{hydroeq} with initial condition $f$.
\end{definition}

Rigorously speaking, above we should not have written $T_t:\mc S\to\mc S$, since we do not know yet if the  image of $T_t$ is contained in $\mc S$. But this is true and it will be proved  below in Corollary \ref{cor42}.

\begin{definition}\label{def:laplacian_operator}
Let $\Delta: \mc S\rightarrow \mc S$   be the Laplacian operator  which is defined on $f\in\mc S$ as 
\begin{equation}\label{laplacian}
\Delta f(u)\;=\;\left\{\begin{array}{cl}
\partial_u^2 f(u)\,,& \mbox{if}\,\,\,u\in(0,1)\,,\\
\partial_u^2 f(0^+)\,,& \mbox{if} \,\,\,u=0\,,\\
\partial_u^2 f(1^-)\,,& \mbox{if} \,\,\,u=1\,.\\
\end{array}
\right.
\end{equation}
Above, $\partial_u^2 f(a^\pm)$ denotes the side limits at the point $a$. The definition of the operator
$\nabla: \mc S\rightarrow C^{\infty}[0,1]$  is analogous.
\end{definition}
We will also use the notations $\p_u$ and $\p_u^2$   for $\nabla$ and $\Delta$, respectively.

\begin{definition}
Let $\mc S'$ be the topological dual of $\mc S$ with respect to the topology generated by the seminorms 
 \begin{equation}\label{semi-norm}
\|f\|_{k}=\sup_{u\in[0,1]}|\partial_u^kf(u)|\,,
\end{equation}
where $k\in\mathbb{N}\cup \{0\}$.
In other words, $\mc S'$ consists of all linear functionals $f:\mc S\to \bb R$ which are continuous with respect to all the seminorms $\Vert \cdot \Vert_k$.
\end{definition}
In order to avoid topological issues we fix once and for all a finite time horizon $T$. 
Let  $\mc D([0,T],\mc S')$ (resp. $\mc C([0,T], \mc S')$) be the space of trajectories which are right continuous, with left limits (resp. continuous) and taking values in $\mc S'$.

\subsubsection{The density fluctuation field}
Fix an initial measure $\mu_n$ in $\Omega_n$. For $x\in\Sigma_n$  and $t\geq 0$, let 
\begin{equation}\label{rho_t}
\rho^n_t(x)\;=\;\mathbb{E}_{\mu_n}[\eta_{tn^2}(x)]\,.
\end{equation}
 We extend this definition to the boundary by setting 
 \begin{equation}\label{ext_bound}
 \rho^n_t(0)\;=\;\alpha\mbox{ and }\rho^n_t(n)\;=\;\beta\,, \mbox{ for all }t\geq 0\,. 
\end{equation}  
 A simple computation shows that $\rho_t^n(\cdot)$ is a solution of the discrete equation given by
\begin{equation}\label{disc_heat}
\left\{
\begin{array}{ll}
 \partial_t \rho_t^n(x) \;= \; \big(n^2\mc B_n \rho_t^n\big)(x)\,, \;\; x\in\Sigma_n\,,\;\;t \geq 0\,,\\
 \rho_t^n(0)\;=\alpha\,, \;\;t \geq 0\,,\\
   \rho^n_t(n)\;=\;\beta\,, \;\;t \geq 0\,,\\
\end{array}
\right.
\end{equation}
 where the operator $\mc B_n$ acts on functions $f:\Sigma_n\cup \{0,n\}\to\bb R$ as
\begin{equation}\label{op_B}
(\mc B_nf)(x)\;=\;\sum_{y=0}^n\xi_{x,y}^n\big(f(y)-f(x)\big)\,, ~~\textrm{ for } x\in \Sigma_n\,,
\end{equation}
where 
\begin{equation*}
\xi_{x,y}^n\;=\; \begin{cases}
 1\,, & \textrm{ if } \; |y-x|=1 \textrm{ and }x,y\in \Sigma_n\,,\\
 \frac{1}{n}\,, & \textrm{ if }\; x=1,y=0\textrm{ and } x=n-1\,, y=n\,,\\ 
 0\,,& \textrm{ otherwise.}
 \end{cases} 
\end{equation*}

\begin{remark}
If  $\mu_n:=\mu_{ss}$ is the non-equilibrium  stationary measure of the system, then the profile  $\rho^n_t(\cdot)$ defined in \eqref{rho_t} is also  stationary in time. In this context, we denote it by   $\rho^n_{ss}(\cdot)$.
As one can see in \cite[Lemma 3.1]{bmns}, $\rho^n_{ss}(\cdot)$ is given by
\begin{equation}\label{stat_profile_disc}
\rho^n_{ss}(x)\;=\;a_n\,x+b_n\quad \textrm{ for }x\in \Sigma_n\,,
\end{equation}
where
$a_n=\pfrac{\beta-\alpha}{3n-2}$ and $b_n=a_n(n-1)+\alpha$.
If we extend the definition of $\rho_{ss}^n(\cdot)$ to the boundary of $\Sigma_n$, as in  \eqref{ext_bound}
we get that $\rho_{ss}^n(\cdot)$ is the stationary solution of \eqref{disc_heat}. 
\end{remark}

Now we define the non-equilibrium  density fluctuation field as follows.
 
\begin{definition}[Density fluctuation field]
We define the density fluctuation field $\mc Y_\cdot^n$ as the time-trajectory of linear functionals acting on functions $f\in\mc S$ as
\begin{equation}\label{density field}
\mc Y^n_t(f)\;=\;\frac{1}{\sqrt{n}}\sum_{x=1}^{n-1}f(\tfrac{x}{n})\Big(\eta_{tn^2}(x)-\rho^n_t(x)\Big)\,.
\end{equation}
\end{definition}

\subsubsection{Non-equilibrium fluctuations}
In the next result we assume the following conditions on the initial state $\mu_n$.

  \begin{assumption}\label{assumption0}
For each $n\in\bb N$, the measure  $\mu_n$ is associated to a measurable profile $\rho_0:[0,1]\to[0,1]$ in the sense of \eqref{eq3}.
\end{assumption}

  \begin{assumption}\label{assumption1}
There exists a constant $C_1>0$ not depending on $n$ such that 
\begin{equation*}
 \max_{ x\in\Sigma_n}\big|\,\rho^n_0(x)-\rho_{0}(\tfrac xn )\big|\;\leq\;\frac{C_1}{n}\,.
\end{equation*}
\end{assumption}
\begin{assumption}\label{assumption2}
There exists a constant $C_2>0$ not depending on $n$ such that for 
\begin{equation}
\label{cor_time_0}\varphi_0^n(x,y)=\mathbb{E}_{\mu_n}[\eta(x)\eta(y)]-\rho_0^n(x)\rho_0^n(y)
\end{equation}
 it holds that
\begin{equation*}
 \max_{1\leq x<y\leq n-1}\big|\varphi_0^n(x,y)\big|\;\leq\;\frac{C_2}{n}\,.
\end{equation*}
\end{assumption}
 For each $n\geq 1$, let  $Q_n$ be the probability measure on $\mc D([0,T],\mc S')$  induced by the density fluctuation field $\mc Y^n_\cdot$ and the measure $\mu_n$.  
\begin{theorem}[Non-equilibrium fluctuations]\label{non_eq_flu}
\quad

The sequence of measures $\{Q_n\}_{ n\in \mathbb{N}}$ is tight on  $\mc D([0,T],\mc S')$ and 
all limit points $Q$ are probability measures concentrated on paths $\mathcal{Y}_\cdot$ satisfying 
 \begin{equation}\label{characterization}
\mathcal{Y}_t(f)\;=\;\mathcal{Y}_0(T_t f)+W_t(f)\,,
\end{equation}
for any $f\in\mathcal S$.
Above $T_t$ is the semigroup  given  in  Definition \ref{def2} and  $W_t(f)$ is a mean zero Gaussian variable of variance  
\begin{equation}\label{eq212}
 \int_0^t\|\nabla T_{t-r}  f\|^2_{L^2(\rho_r)}dr\,,
\end{equation}
where for $r>0$ 
\begin{equation}\label{norm}
\begin{split}
\<f, g\>_{L^2(\rho_r)}=&\,\big[\alpha+(1-2\alpha)\rho(r,0)\big]\,f(0)g(0)+ \big[\beta+(1-2\beta)\rho(r,1)\big]\,f(1)g(1)\\
+\;& \int_0^12\chi(\rho(r,u))\,f(u)g(u)\,du\,, \\
\end{split}
\end{equation}
  $\rho(t,u)$ is the solution of the hydrodynamic equation  \eqref{hydroeq}, and $\chi(u)=u(1-u)$. 
Moreover,  $\mathcal{Y}_0$ and $W_t$ are uncorrelated in the sense that $\bb E_Q \Big[\mc Y_0(f) \, W_t(g) \Big]\;=\;0$
for all $f,g\in \mc S$.
\end{theorem}

\begin{theorem}[Ornstein-Uhlenbeck limit]\label{thm27}
\quad

Assume that the sequence of  initial density fields $\{\mc{Y}_0^n\}_{n\in\bb N}$ converges, as $n\to\infty$, to a mean-zero Gaussian field  $\mc Y$ with covariance given on $f,g\in\mc S$ by
\begin{equation}\label{covar}
\lim_{n\to\infty}\mathbb E_{\mu_n}\Big[\mc Y^n_0 (f)\mc Y^n_0(g)\Big]\;=\;\mathbb E\,\Big[\mc Y (f)\mc Y(g)\Big]\;:=\;\sigma(f,g)\,.
\end{equation}
Then, the sequence $\{Q_n\}_{ n\in \mathbb{N}}$ converges, as $n\to\infty$, to a generalized Ornstein-Uhlenbeck (O.U.) process, which is the formal solution of  the equation:
\begin{equation} \label{O.U.}
\partial_t \mathcal{Y}_t\;=\;\Delta\mathcal{Y}_tdt+\sqrt{2\chi(\rho_t)}\nabla W_t\,,
\end{equation}
where $ W_t$ is a Brownian Motion of unit variance and $\Delta$,  $\nabla$ are given in Definition \ref{def:laplacian_operator}. As a consequence, the covariance of the limit field $\mathcal{Y}_t$ is given on $f,g\in{\mc S}$ by
\begin{equation}\label{covariance non eq limit field}
 E\,[\mathcal{Y}_t(f)\mathcal{Y}_s(g)]\;=\;\sigma(T_tf,T_sg)+\int_0^s\<\nabla T_{t-r} f, \nabla T_{s-r}g\>_{L^2(\rho_r)}dr\,.
\end{equation}
\end{theorem}

In Subsection \ref{sub42} we present the precise definition of such generalized O.U. process.
As a consequence of the previous result we obtain the 
non-equilibrium fluctuations starting from a Local Gibbs state.

\begin{corollary}[Local Gibbs state]\label{cor}

Fix a Lipschitz  profile  $\rho_0:[0,1]\to[0,1]$
and suppose to start the process from a Bernoulli product measure given by $\mu_n\{\eta:\eta(x)=1\}=\rho_0(\tfrac{x}{n})$.
Then, the Theorem \ref{thm27} remains in force and the covariance in this case is given on $f,g\in\mc S $ by 
\begin{equation}\label{covariance_local_gibbs}
 E\,[\mathcal{Y}_t(f)\mathcal{Y}_s(g)]\;=\;\int_0^1 \chi(\rho_0(u))\,T_tf(u)T_sg(u)\,du+\int_0^s\<\nabla T_{t-r} f, \nabla T_{s-r}g\>_{L^2(\rho_r)}dr\,,
\end{equation}
where $\rho(t,u)$ is the solution of the hydrodynamic equation \eqref{hydroeq} with initial condition given by $\rho_0(\cdot)$.
\end{corollary}

From Theorem \ref{thm27} to prove the last result,  it is enough to show the convergence at the initial time, that is:
\begin{equation*}
\lim_{n\to\infty}\mathbb E_{\mu_n}\Big[\mc Y^n_0 (f)\mc Y^n_0(g)\Big]\;=\;\int_0^1 \chi(\rho_0(u))\,f(u)g(u)\,du\,,
\end{equation*} 
which can be easily verified by means of the convergence of characteristic functions, in the same way of \cite[page 297, Cor. 2.2]{kl}. We leave the details to the reader.

\subsubsection{Stationary fluctuations}

Fix $\alpha\neq \beta$. Consider the process starting from the stationary measure $\mu_{ss}$. Note that 
the density fluctuation field defined on \eqref{density field} is simply given on  $f\in\mc S$ by
\begin{equation}\label{density field_stat}
\mc Y^n_t(f)\;=\;\frac{1}{\sqrt{n}}\sum_{x=1}^{n-1}f(\tfrac{x}{n})\Big(\eta_{tn^2}(x)-\rho_{ss}^n(x)\Big)\,,
\end{equation}
where $\rho_{ss}^n(x)$ is defined in \eqref{rho_t} with $\mu_n=\mu_{ss}$ and given explicitly in \eqref{stat_profile_disc}.

\begin{theorem}[Stationary fluctuations]\label{flustat}
\quad

Suppose to start the process from $\mu_{ss}$ with $\alpha\neq{\beta}$. Then, $\mathcal{Y}^n$ converges to the centered Gaussian field $\mathcal{Y}$  with covariance given on $f,g\in\mc S$ by:
\begin{equation}\label{stat convariance}
\begin{split}
& E_{\mu_{ss}}[\mathcal{Y}(f)\mathcal{Y}(g)]\;=\;\int_0^1 \chi(\overline{\rho}(u))f(u)g(u)\,du-\Big(\frac{\beta-\alpha}{3}\Big)^2\int_0^1 [(-\Delta)^{-1}f(u)]g(u) \;du\\
&+\frac{2(2\beta+\alpha)(2\beta-1)}{3}\int_0^\infty\hspace{-10pt} T_t f(1)T_t g(1)\,dt+\frac{2(\beta+2\alpha)(2\alpha-1)}{3}\int_0^\infty\hspace{-10pt} T_t f(0)T_t g(0)\,dt\,,\\
\end{split}
\end{equation}
with $\overline{\rho}(u)=\big(\pfrac{\beta-\alpha}{3}\big)\,u +\pfrac{\beta+2\alpha}{3}$, which is the stationary solution of \eqref{hydroeq}.
\end{theorem}
The time integrals above are well defined in view of the fast decaying of the semigroup $T_t$, see  Corollary \ref{cor421} below.

We  interpret the covariance formula  above in the following way: the first term at the right hand side of \eqref{stat convariance} corresponds to the covariance associated to  $\overline{\rho}$ in the bulk; the second term corresponds to the covariance associated to  $\p_u \overline{\rho}=(\beta-\alpha)/3$ also in the bulk. The third and fourth terms are associated to  $\overline{\rho}$ at the boundaries. Note that for the particular value $\alpha=1/2$ (or $\beta=1/2$) the corresponding boundary term vanishes. 
\section{Semigroup results}\label{sec:sem}
In this section we present some useful  results about the hydrodynamic equation. We start with the homogeneous version of  \eqref{hydroeq}, i.e., considering $\alpha=\beta=0$ as displayed below:  \begin{numcases}{}\label{hom_hydroeq}
\p_t \rho(t,u)= \p_{u}^2 \rho(t,u)\,, & \textrm{ for } $t>0$\,,\, $u\in (0,1)$\,,\\
 \p_u \rho(t,0) = \rho(t,0)\,, & \text{for } $t>0$\,, \label{b1}\\
 \p_u \rho(t,1) = -\rho(t,1)\,, & \textrm{for } $t>0$\,, \label{b2} \\
 \rho(0,u)= \rho_0(u)\,,& $u\in [0,1]$\,. \label{b3}
\end{numcases}

\begin{proposition}\label{prop23}
Suppose that $\rho_0\in L^2[0,1]$. Then the previous equation has a solution given by
\begin{equation}\label{eq6semi}
(T_t \rho_0)(u) \;:=\; \sum_{n=1}^\infty a_n\,e^{-\lambda_n t}\,\Psi_n(u)\,,
\end{equation}
where $\{\Psi_n\}_{n\in \bb N}$ is an orthonormal basis of  $L^2[0,1]$ constituted by  eigenfunctions  of the associated  Regular Sturm-Liouville problem (see \eqref{eq34}-\eqref{eq36} below), and $a_n$ are the Fourier coefficients of $\rho_0$ in that basis. In particular, we prove that  $\lambda_n\sim n^2\pi^2$, and as a consequence,  the series \eqref{eq6semi} converges exponentially fast, implying  that $(T_t\rho_0)(u)$ is smooth in space and time for $t>0$. 
\end{proposition}
\begin{proof}
We start the proof with the associated \textit{Regular Sturm-Liouville Problem}, for details on this subject we refer to \cite{birkhoff89}, for instance.
 For $\lambda\in\bb R$, consider the following second-order ordinary differential equation:
 \begin{numcases}{}
   \Psi''(u)+\lambda \Psi(u) \;=\; 0\,, \quad u\in(0,1)\,, &  \label{eq34}   \\
  \Psi(0)\;=\;\Psi'(0)\,, &  \label{eq35}   \\
  \Psi(1)\;=\;-\Psi'(1)\,. &  \label{eq36}   
\end{numcases}
For $\lambda\leq 0$  there is no solution except the trivial one. For $\lambda>0$, the general solution of 
\eqref{eq34} is of the form
$\Psi(u)\;=\;A\sin (\sqrt{\lambda} \, u )+ B \cos(\sqrt{\lambda} \,u)\,.$
Then,  the boundary condition \eqref{eq35} leads to 
$B\;=\; \sqrt{\lambda}A\,.$
To avoid the null solution, we henceforth impose $A\neq 0$.
On the other hand, the boundary condition \eqref{eq36} gives
\begin{equation*}
A\sin(\sqrt{\lambda}) + A\sqrt{\lambda}\cos (\sqrt{\lambda})\; =\;
-\,\Big[A\sqrt{\lambda}\cos(\sqrt{\lambda}) - A(\sqrt{\lambda})^2\sin (\sqrt{\lambda})\Big]\,,
\end{equation*}
which becomes the transcendental equation
\begin{equation}\label{trans}
\tan (\sqrt{\lambda}) \;=\; \frac{2\sqrt{\lambda}}{\lambda-1}\,,
\end{equation}
for which there exists a countable number of solutions. 
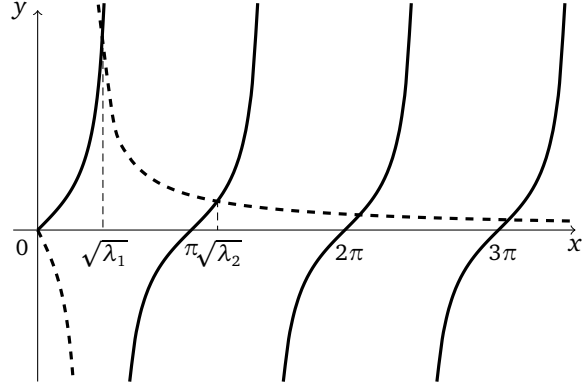
\begin{figure}[!htb]
\centering
\begin{tikzpicture}[scale=0.65,smooth];
\draw[->] (0,-3.1)--(0,4.5) node[anchor=east]{$y$};
\draw[->] (-0.5,0)--(11,0) node[anchor=north]{$x$};

\draw[black = solid,  very thick ] plot   [domain=0:5*pi/12+0.05](\x,{tan(\x r)});

\draw[black = solid,  very thick  ] plot    [domain=7*pi/12+0.05:17*pi/12+0.05](\x,{tan(\x r)});

\draw[black = solid,  very thick  ] plot   [domain=19*pi/12+0.05:29*pi/12+0.05](\x,{tan(\x r)});
\draw[black = solid,  very thick  ] plot   [domain=31*pi/12+0.05:41*pi/12+0.05](\x,{tan(\x r)});

\draw (0,0) node[anchor=north east]{\small $0$};
\draw (pi,-0.1) node[anchor=north]{\small $\pi$};
\draw (2*pi+0.1,-0.1) node[anchor=north]{\small $2\pi$};
\draw (3*pi+0.1,-0.1) node[below]{\small $3\pi$};

\draw [dashed, very thick, domain=0:0.72] plot (\x, {2*\x*pow(\x*\x-1,-1)}); 
\draw [dashed, smooth, very thick, domain=1.24:9] plot (\x, {2*\x*pow(\x*\x-1,-1)}); 
\draw [dashed, smooth, very thick, domain=9.1:11] plot (\x, {2*\x*pow(\x*\x-1,-1)}); 

\draw[densely dashed] (1.335,4) -- (1.335,0) node[below]{\small $\sqrt{\lambda_1}$}; 
\draw[densely dashed] (3.68,0.63) -- (3.68,0) node[below]{\small $\sqrt{\lambda_2}$};
\end{tikzpicture}
\bigskip

\caption{In solid line, the graph of $f(u) = \tan (u)$. In dashed line, the graph of $g(u)=2u/(u^2-1)$.}\label{fig3}
\end{figure}

 For each $n\in\bb N$, let $\lambda_n$ be the solution of \eqref{trans} satisfying $(n-1)\pi\leq \sqrt{\lambda_n}\leq n\pi$, see  the Figure~\ref{fig3}. Thus $0<\lambda_1<\lambda_2<\lambda_3<\cdots$ and $\lambda_n\sim n^2\pi^2$ as $n\to\infty$. Denote now
\begin{equation}\label{eigen42}
\Psi_n(u)\;=\;A_n\sin (\sqrt{\lambda_n} \,u )+ A_n\sqrt{\lambda_n} \cos(\sqrt{\lambda_n} \, u)\,,
\end{equation} 
where $A_n$ is a normalizing constant in such a way $\Psi_n$ has unitary $L^2[0,1]$-norm.

    Note that $\Psi_n$ is the solution of the Sturm-Liouville problem above associated to the eigenvalue $-\lambda_n$. Moreover, the set $\{\Psi_n\}_{n\in \bb N}$ is an orthonormal basis of  $L^2[0,1]$.  The  orthogonality comes  from the fact the associated Sturm-Liouville operator is self-adjoint and a proof of it can be found  in \cite[page 303, Thm 1]{birkhoff89}. The proof of completeness can be found in \cite[page 363, Section 9]{birkhoff89}.

Since now we have an orthonormal basis of eigenfunctions, we will apply the heuristics of the classical  method of separation of variables to obtain the solution of  \eqref{hom_hydroeq}-\eqref{b3}. We start supposing that the solution has the form $\rho(u,t)=\Psi(u)\Phi(t)$. By \eqref{hom_hydroeq} we get that
\begin{equation*}
\frac{\Phi'(t)}{\Phi(t)}\;=\;-\lambda\;=\;\frac{\Psi''(u)}{\Psi(u)}\,,
\end{equation*} 
for some $\lambda\in \bb R$. The first equality in the previous display gives $\Phi(t)=ce^{-\lambda t}$ for some $c\in \bb R$, while the second equality, under the boundary conditions in \eqref{b2} and \eqref{b3}, corresponds to the Sturm-Liouville problem stated above. Therefore, the  solution we seek  will be of the form
$c e^{-\lambda_n t}\Psi_n(u)$. Assuming that $\rho_0\in L^2[0,1]$, the completeness of the basis permits to write
\begin{equation*}
\rho_0(u)\;=\;\sum_{n=1}^\infty a_n\Psi_n(u)\,,
\end{equation*}
where $a_n=\<\rho_0\,,\Psi_n\>$ are the Fourier coefficients, and $\<\cdot,\cdot\>$  is the inner product w.r.t. $L^2[0,1]$. Now it is a simple task to check that the solution of \eqref{hom_hydroeq}-\eqref{b3} is given by
$\rho(t,u)\;=\;\sum_{n=1}^\infty a_n\,e^{-\lambda_n t}\,\Psi_n(u)$.
This completes the proof.
\end{proof}

The previous proposition implies the next results, which play an important role when showing the uniqueness of the associated generalized Ornstein-Uhlenbeck process of Theorem \ref{thm27}.

\begin{corollary}\label{cor42}
If $f\in \mc S$, then for any $t>0$, we have that $ T_t f \in \mc S$ and  $\Delta T_t f\in \mc S$.
\end{corollary}
\begin{proof}
Since  $\Psi_n$ is a linear combination of sine and cosine, and the conditions  of Definition~\ref{def1} are linear, then $\Psi_n\in \mc S$ for all $n\in \bb N$. 
This property is inherited by $T_t f$ and $\Delta T_t f$ due to the explicit formula \eqref{eq6semi} and its exponential convergence. 
\end{proof}

\begin{corollary}\label{cor421}
 For any $f\in \mc S$, we have
$\displaystyle\lim_{t\to\infty} T_t f \;=\;0
$
 in the supremum norm. Moreover, this convergence is exponentially fast.\end{corollary}
\begin{proof}
Immediate  from \eqref{eq6semi}.
\end{proof}

\begin{corollary}\label{corinv} The operator $\Delta :\mc S\to\mc S$ is a bijection. Moreover, for any $f\in \mc S$,
\begin{equation*}\begin{split}
&(a)\quad (-\Delta)^{-1} f(u)\;=\; \int_0^{+\infty}T_t f(u) \,dt\,,\\
&(b)\quad \lim_{t\to\infty} \int_0^t\int_0^1 2 \,\big(T_{r}f (u)\big)^2\,du\,dr\;=\;\int_0^1 f(u)(-\Delta)^{-1}f(u)\,du\,.
\end{split}\end{equation*}
\end{corollary}
\begin{proof}
First of all, notice that all the time integrals above are well defined due to the Corollary~\ref{cor421}. We start by showing $(a)$. Let $f\in \mc S$ and write
$f\;=\; \sum_{n=1}^\infty a_n \Psi_n$.
Then, by Proposition \ref{prop23},
\begin{equation*}
\begin{split}
& \Delta \int_0^\infty T_t f (u) \,dt  \;=\;\Delta \int_0^\infty \Big(\sum_{n=1}^\infty a_n e^{-\lambda_n t} \Psi_n(u) \Big)\,dt\;=\;
\Delta \sum_{n=1}^\infty \Big(\int_0^\infty  a_n e^{-\lambda_n t} \Psi_n(u) \Big)dt\\
&=\; \Delta\sum_{n=1}^\infty  \big(\frac{a_n}{\lambda_n}\big) \, \Psi_n(u) \;=\; \sum_{n=1}^\infty  \big(\frac{a_n}{\lambda_n}\big) \Delta\, \Psi_n(u) \;=\;- \sum_{n=1}^\infty  a_n \Psi_n(u) \;=\;-f(u)\,,
\end{split}
\end{equation*}
where in the penultimate equality we have used the fact that $-\lambda_n$ is the eigenvalue associated to the eigenfunction $\Psi_n$.
Now we prove $(b)$. 
Since $\{\Psi_n\}_{n\in \bb N}$ is an orthonormal basis of  $L^2[0,1]$, again by Proposition \ref{prop23} we have that
$\< T_rf, T_r f\>=\; \< f,T_{2r}f\>$. Therefore,
\begin{equation*}
\int_0^t\int_0^1 2 \, \big(T_{r}f (u)\big)^2\,du\, dr\;=\;\int_0^t\int_0^12 \, f (u) T_{2r} f (u)\,du\, dr\;=\;\int_0^{2t}\int_0^1 f (u) T_{r} f (u)\,du\, dr.
\end{equation*}
Now, first apply Fubini's Theorem, then take the limit as $t\to\infty$ and recall $(a)$. This finishes  the  proof.
\end{proof}

\begin{corollary} \label{cor43}
Let  $\overline{\rho}(\cdot)$ be the stationary solution of \eqref{hydroeq} disregarding the initial condition. Then, the solution $\rho(t,\cdot)$ of \eqref{hydroeq} is given by
\begin{equation*}
\rho(t,u)\;=\;  \overline{\rho}(u) +T_t\big( \rho_0-\overline{\rho}\big)\,(u)\,,\quad u\in [0,1]\,,\; t\geq 0\,,
\end{equation*}
 where $T_t$ is the semigroup previously described. In particular, the solution $\rho(t,\cdot)$ of \eqref{hydroeq} is smooth in space and time.
\end{corollary}
\begin{proof}
First we note that from \cite[Theorem 2.2]{bmns} we have, for $u\in [0,1]$, that
\begin{equation*}
\overline{\rho}(u)\;=\; \Big(\frac{\beta-\alpha}{3}\Big)u +\Big(\alpha+\frac{\beta-\alpha}{3}\Big)\,.
\end{equation*}
 Then, the time derivative of the function $\overline{\rho}(\cdot)$ is null,  as well as its second derivative in space. On other hand, $\overline{\rho}(\cdot)$ satisfies the required non-homogeneous boundary conditions. The result then easily  follows by a  direct verification.
\end{proof}
\begin{corollary}\label{cor45}
Let  $\rho(t,\cdot)$ be the solution of \eqref{hydroeq}. Then, for any $u\in[0,1]$,
\begin{equation*}
\lim_{t\to \infty} \rho(t,u)\;=\;  \overline{\rho}(u)\,,
\end{equation*}
in the supremum norm.
\end{corollary}
\begin{proof}
Immediate from the two previous corollaries.
\end{proof}

\section{Proof of Theorem \ref{non_eq_flu}}\label{s3}
In this section we prove Theorem \ref{non_eq_flu}  which is the main result of this paper.  We start with the martingale decomposition of the process.
\subsection{Martingale decomposition}
Let $\phi:[0,T]\times [0,1]\to \bb R$ be a test function. We refer to \cite[p. 330]{kl} for a proof that
\begin{align}
M^n_t(\phi)& \;:=\;\mc Y_t^n(\phi)-\mc Y_0(\phi)-\int_{0}^t \Lambda^n(\phi)\,ds\,,\label{martingaleM} \\
N^n_t(\phi)& \;:=\;(M_t^n(\phi))^2-\int_{0}^t \Gamma^n_s(\phi)\,ds\label{quadratic}
\end{align}
are martingales with respect to the natural filtration $\mc F_t:=\sigma(\eta_s:\, s\leq t)$, where
\begin{equation*}
\begin{split}
& \Lambda^n(\phi)\;:=\; (\partial_s+n^2\mc L_n)\mc Y_s^n(\phi)\,,\\
& \Gamma^n_s(\phi)\;:=\;n^2\mc L_n \mc Y_s^n(\phi)^2-2\mc Y_s^n(\phi)n^2\mc L_n \mc Y_s^n(\phi)\,.
\end{split}
\end{equation*}
  By a long but elementary computation, 
\begin{equation}\label{int part of mart}
\begin{split}
\Lambda^n(\phi)\;=\;& \mc Y_s^n(\p_s \phi)+\frac{1}{\sqrt{n}}\sum_{x=1}^{n-1}\Delta_n \phi(\tfrac{x}{n}) \Big(\eta_{sn^2}(x)-\rho^n_s(x)\Big)\, ds\\
&+ \sqrt n\, \Big[\nabla_n^+\phi(0)-\phi\big(\pfrac{1}{n}\big)\Big]\Big(\eta_{sn^2}(1)-\rho_s^n(1)\Big)\\
&+\sqrt n\,\Big[\phi(\tfrac{n-1}{n})+\nabla_n^-\phi(1)\Big]\Big(\eta_{sn^2}(n-1)-\rho_s^n(n-1)\Big)\,.
\end{split}
\end{equation}
Note that the second term at the right hand side of the previous expression is  $\mc Y_s^n(\Delta_n \phi)$.
Above,  we have used the notation 
\begin{equation*}
\nabla_n^+\phi(x)\;=\;n\,\Big[\phi(\tfrac{x+1}{n})-\phi(\tfrac{x}{n})\Big]\quad 
\textrm{and}
\quad  \nabla_n^-\phi(x)\;=\;n\,\Big[\phi(\tfrac{x}{n})-\phi(\tfrac{x-1}{n})\Big]\,.
\end{equation*}
Also by direct computations we get that
\begin{equation}\label{quad var}
\begin{split}
& \Gamma^n_s(\phi)\;=\;\frac{1}{n}\sum_{x=1}^{n-2}\Big(\nabla^+_n \phi\big(\pfrac{x}{n}\big)\Big)^2 \Big(\eta_{sn^2}(x)-\eta_{sn^2}(x+1)\Big)^2\\
&+ \Big(\phi\big(\pfrac{1}{n}\big)\Big)^2\Big(\alpha-2\alpha \eta_{sn^2}(1)+\eta_{sn^2}(1)\Big)\;+\; \Big(\phi\big(\pfrac{n-1}{n}\big)\Big)^2\Big(\beta-2\beta \eta_{sn^2}(n-1)+\eta_{sn^2}(n-1)\Big)\,.
\end{split}
\end{equation}

\subsection{Proof of  Theorem \ref{non_eq_flu}}

\begin{lemma}\label{lemma32}
For $\phi\in\mc S$, the sequence of martingales $\{M^n_t(\phi);t\in [0,T]\}_{n\in \bb N}$ converges in the topology of $\mc D([0,T], \bb R)$, as $n\to\infty$, towards a Brownian motion $W_t(\phi)$ with  quadratic variation given by 
\begin{equation}\label{36}
\begin{split}
  \int_0^t \Bigg\{\int_0^12\chi(\rho(r,u)) \big(\nabla \phi(u)\big)^2du  \;+\;&\big[\alpha+(1-2\alpha)\rho(r,0)\big]\big(\phi(0)\big)^2\\
 &+ \big[\beta+(1-2\beta)\rho(r,1)\big]\big(\phi(1)\big)^2\Bigg\}\,dr\,,
\end{split}
\end{equation}
where $\rho(t,u)$ is the solution of the hydrodynamic equation \eqref{hydroeq}.
\end{lemma}
\begin{proof}
To prove this lemma we  can apply  Proposition 4.2 of \cite{GJ2016}, which is based on  three hypotheses. In our case, the first hypothesis of that proposition is trivially verifiable by the expression of the quadratic variation. The second hypothesis can be proved as in \cite{GJ2016} by noting that when a jump occurs, the configuration $\eta$ only changes its value in at most two sites, hence the same estimate as in \cite{GJ2016} holds here. Finally, we prove that the  third hypothesis  holds.
In view of \eqref{quad var} and recalling Assumption \ref{assumption0},  we may apply   \cite[Lemma 5.7]{bmns} and \cite[Thm 2.7]{bmns} and standard arguments to conclude that
\begin{equation}\label{35}
\int_{0}^t\Gamma^n_s(\phi)ds\,,
\end{equation}
 which is an \textit{additive functional} of the exclusion process $\eta_t$,
converges in distribution, as $n\to\infty$, towards \eqref{36}. Moreover, since the expression above is deterministic, the convergence holds, in fact,  in probability.
This finishes the proof.
\end{proof}
\begin{remark}\label{remark31}\rm
We  point out that  expression \eqref{36} simply   writes as
\begin{equation*}
\int_0^t \|\nabla \phi\|_{L^2(\rho_r)}^2\,dr\,,
\end{equation*}
provided $\phi\in \mc S$ ($\phi(0)=\nabla\phi(0)$ and $\phi(1)=-\nabla\phi(1)$) for any time $r\in[0,T]$, see \eqref{norm}. 
\end{remark}

As a consequence of the previous result, for each $t$, the random variable  $W_t(\phi)$ is Gaussian with mean zero and with variance  $\int_0^t \|\nabla \phi\|_{L^2(\rho_r)}^2dr$.
\medskip

Moreover,  the random variables $W_t(f)$ and $\mc Y_0(g)$ are uncorrelated for any $f,g\in \mc S$. In fact,
$\bb E\,\big[ W_t(f) \mc Y_0(g)\big] 
=\bb E\, \Big[ \mc Y_0(g)\;\bb E\big[ W_t(f) |\mc F_0\big]\Big]=0$, since $W_0(f)=0$.
\medskip

The proof of tightness is postponed to Section \ref{s6}.
 From this point on (in this subsection), fix $t\in [0,T]$ and restrict the processes to the time interval $[0,t]$.
Choose now the particular test function 
\begin{equation}\label{test}
\phi(u,s)\;:=\;(T_{t-s}f)(u)\,,
\end{equation} where
$T_t$ is given in  Definition \ref{def2} and  $f\in \mc S$. Note that $\phi$ is well-defined for all $s\in[0,t]$ and that  $\phi\in \mc S$ in view of Corollary \ref{cor42}.
For this choice of the test function,  \eqref{int part of mart} writes as 
\begin{equation}\label{3.3}
\begin{split}
 \Lambda^n( T_{t-s}f)   \;=\; & \mc Y_s^n\Big(\Delta_n T_{t-s}f  - \Delta T_{t-s}f \Big)  + \mc Y_s^n\Big(\Delta T_{t-s}f + \p_s T_{t-s}f \Big)  \\
&+ \sqrt n\,\Big( \nabla_n^+(T_{t-s}f)(0) -(T_{t-s}f)\big(\pfrac{1}{n}\big)\Big)\cdot\big(\eta_{sn^2}(1)-\rho_s^n(1)\big)\\
&+\sqrt n\,\Big( \nabla_n^-(T_{t-s}f)(1)-(T_{t-s}f)\big(\pfrac{n-1}{n}\big)\Big)\cdot \big(\eta_{sn^2}(n-1)-\rho_s^n(n-1)\big)\,.\\
\end{split}
\end{equation}

We claim that $\Lambda^n(T_{t-s}f)$ goes to zero as $n\to\infty$.
Let us examine the four terms at the right hand side of \eqref{3.3}.
By Proposition \ref{prop23}, we know that $T_{t-s}f$ is smooth, hence
$\Delta_n T_{t-s}f  - \Delta T_{t-s}f$ is of order $O(n^{-2})$. The second term at the right hand side of \eqref{3.3} is identically zero, since $\p_s T_{t-s}f=-\Delta T_{t-s}f$. The third and fourth terms at the right hand side of \eqref{3.3} go to zero as $n\to\infty$, because $T_{t-s}f$ satisfies the boundary conditions  \eqref{b1} and \eqref{b2}, respectively. This proves  the claim, which  implies also that $\int_0^t \Lambda_n(T_{t-s}f)\,ds$ goes to zero. In other words, by choosing   \eqref{test}, the integral term in \eqref{martingaleM} vanishes in the limit as $n\to\infty$.

Let us now look at  \eqref{martingaleM} for a fixed time $t\in[0,T]$ and for the choice \eqref{test}. As a consequence of the previous results, together with tightness which is proved in Section~\ref{s6}, any limit point of the sequence  $\{\mc Y_t^n(f)\}_{n\in\bb N}$   must be of the form
\begin{equation}\label{eq38}
\mc Y_t(f)  \;=\;   \mc Y_0(T_t f)+ W_t(f)\,,
\end{equation}
where $\mc Y_0(T_t f)$ and $W_t(f)$ are uncorrelated and $W_t(f)$ is a mean zero Gaussian variable of variance given by  \eqref{eq212}. This finishes the proof of Theorem~\ref{non_eq_flu}.

\section{Proof of Theorem \ref{thm27}}\label{s4}

In this section we start by showing the uniqueness of the Ornstein-Uhlenbeck process solution of \eqref{O.U.} by a martingale problem and then we prove the Theorem \ref{thm27}.
\subsection{Uniqueness of the Ornstein-Uhlenbeck process}\label{sub42}

\begin{proposition}\label{pp1}
There exists an unique random element $\mc Y$ taking values in the space $\mc C([0,T],\mathcal{S}')$ such
that:
\begin{itemize}
\item[(i)] For every function $f \in \mathcal{S}$, 
\begin{align}
& W_t(f)\;:=\; \mc Y_t(f) -\mc Y_0(f) -  \int_0^t \mc Y_s(\Delta f)ds\,,\label{martM}\\
& N_t(f)\;:=\;\big( W_t(f)\big)^2 - \int_0^t \|\nabla f\|_{L^2(\rho_r)}^2\,dr\label{martN}
\end{align}
are martingales with respect to the  filtration $\mc F_t:=\sigma(\mc Y_s(g); s\leq t,  g \in \mathcal{S})$.
\item[(ii)] $\mc Y_0$ is a Gaussian field of mean zero and covariance given on $f,g\in\mathcal{S}$ by
\begin{equation}\label{eq:covar1}
\mathbb{E}\big[ \mc Y_0(f) \mc Y_0(g)\big] \;=\;  \sigma(f,g)\,,
\end{equation}
where $\sigma$ was defined in \eqref{covar}.
\end{itemize}
Under the conditions above we have that:  for each $f\in\mc S$, the  process $\{\mc Y_t(f)\,;\,t\geq 0\}$ is  Gaussian. Moreover, for $s<t$  the distribution of $\Y_t(f)$  conditionally to
$\mc F_s$ is normal  of mean $\Y_s(T_{t-s}f)$ and variance $\int_s^{t}\Vert \nabla T_{t-r}
f\Vert^2_{L^2(\rho_r)}dr$.
\end{proposition}
Before proving the proposition we make some comments. The existence of the random element $\mc Y$  is a consequence of tightness, which is proved in Section \ref{s6}. The fact that \eqref{martM} and \eqref{martN} are martingales motivates us to call  the random element $\mc Y$  the \textit{formal} solution of \eqref{O.U.}. From this formal equation \eqref{O.U.} the random element $\mc Y$ coins the name 
\textit{generalized Ornstein-Uhlenbeck}. We strongly emphasize that equation \eqref{O.U.} is solely formal and that the norm $\|\cdot\|_{L^2(\rho_r)}$ plays a role in the definition of this generalized Ornstein-Uhlenbeck process.
\medskip

The next lemma is the key in the proof of Proposition \ref{pp1}.
\begin{lemma}\label{lemma44} For any $f\in\mc S$,
$T_{t+\eps} f-T_t f=\eps\, \Delta T_t H+o(\eps,t)$, where  $o(\eps,t)$ denotes a function in $\mc S$ such that
$\lim_{\eps \searrow 0}  \frac{o(\eps,t)}{\eps} \;=\;0
$
holds in the topology of $\mc S$. Moreover,  the limit is uniform in compact time intervals. 
\end{lemma}
\begin{proof}
The proof is a direct consequence of the explicit formula \eqref{eq6semi}. Notice that the inclusion $o(\eps,t)\in \mc S$ is immediate from Corollary \ref{cor42}. Details are omitted here.
\end{proof}
\begin{proof}[Proof of Proposition \ref{pp1}]
 The structure of proof is the same of  \cite[page 307]{kl}. 
 By \eqref{martM} and \eqref{martN}, we have that
 \begin{equation*}
 W_t(f)\cdot\Bigg(\frac{1}{t}\int_0^t \| \nabla f\|_{L^2(\rho_r)}dr\Bigg)^{-\frac{1}{2}}
 \end{equation*}
 is a standard Brownian motion.
 
 Fix $f\in\mc S$ and $s>0$.
 By It\^o's Formula (see \cite[Thm. 3.3 and Cor. 3.3]{ry}) and the previous comment,  
 the process $\{X_t^s(f)\,;\,t\geq s\}$ defined by
\begin{equation*}
 X_t^s(f)\;=\;\exp\Bigg\{\frac{1}{2}\int_s^t\|\nabla f\|_{L^2(\rho_r)}^2dr+ i\,\Big( \Y_t(f) - \mc Y_s(f) -\int_s^t\Y_r(\Delta 
f)\,dr\Big)\Bigg\}
\end{equation*}
is a (complex) martingale. Fix $S>0$. We claim now that the process $\{Z_t\,;\, 0\leq t\leq S\}$  defined by
\begin{equation*}
 Z_t\;=\;\exp\Bigg\{ \frac{1}{2}\int_0^t\Vert \nabla T_{S-r} f\Vert^2_{L^2(\rho_r)}\,dr +i\,\Y_t(T_{S-t} f)\Bigg\}
\end{equation*}
is also a complex martingale. To prove this claim, consider two times $0\leq t_1<t_2\leq S$ and a partition of the interval $[t_1,t_2]$ in $n$
intervals of equal size, that is, $t_1=s_0<s_1<\cdots<s_n=t_2\,,$
with $s_{j+1}- s_j=(t_2-t_1)/n$. Observe  that
\begin{equation*}
\begin{split}
\prod_{j=0}^{n-1} X_{s_{j+1}}^{s_j}(T_{S-s_j}f)\;=\;&\exp\Bigg\{  \sum_{j=0}^{n-1}\,\frac{1}{2}\int_{s_j}^{s_{j+1}}
\Vert \nabla T_{S-s_j}f\Vert^2_{L^2(\rho_r)}dr\\
& +i\,\sum_{j=0}^{n-1}
\Big( \Y_{s_{j+1}}(T_{S-s_j}f) - \Y_{s_j}(T_{S-s_j}f) -\int_{s_j}^{s_{j+1}}\Y_r(\Delta
T_{S-s_j}f)\,dr\Big)\Bigg\}\,.\\
\end{split}
\end{equation*}
As $n\to +\infty$, the first sum inside the exponential above converges to
\begin{equation*}
\frac{1}{2}\int_{t_1}^{t_2}\Vert \nabla T_{S-r}f\Vert^2_{L^2(\rho_r)} \,dr\,,
\end{equation*}
due to  the smoothness of the semigroup $T_t$. The second sum inside the exponential can be rewritten as
\begin{equation*}
 \Y_{t_2}(T_{S-t_2+\frac{1}{n}}f)-\Y_{t_1}(T_{S-t_1}f)+ \sum_{j=1}^{n-1}
\Bigg( \Y_{s_{j}}(T_{S-s_{j-1}}f-T_{S-s_j}f) -\int_{s_j}^{s_{j+1}}\!\!\Y_r(\Delta
T_{S-s_j}f)\,dr\Bigg)\,.
\end{equation*}
Using the fact that $\Y\in \mc C([0,T],\mc S')$ and Lemma \ref{lemma44},
we obtain that the  previous expression converges almost surely to 
$\Y_{t_2}(T_{S-t_2}f)-\Y_{t_1}(T_{S-t_1}f)\,.$
Hence we have proved that
\begin{equation*}
 \lim_{n\to {+\infty}} \prod_{j=0}^{n-1} X_{s_{j+1}}^{s_j}(T_{S-s_j}f)\,=\,\exp\Bigg\{
 \frac{1}{2}\!\int_{t_1}^{t_2}\!\!\!\!\Vert \nabla T_{S-r}f\Vert^2_{L^2(\rho_r)} dr+
 i\Big(\Y_{t_2}(T_{S-t_2}f)-\Y_{t_1}(T_{S-t_1}f)\Big)\Bigg\},
\end{equation*}
 which is equal to $\frac{ Z_{t_2}}{Z_{t_1}}$ almost surely. Since the complex exponential is bounded, the Dominated Convergence
Theorem gives additionally  the $L^1$
convergence, which  implies
\begin{equation*}
 \bb E\,\Big[G\,\frac{Z_{t_2}}{Z_{t_1}} \Big] \;=\;\lim_{n\to {+\infty}} \bb E\,\Big[G\,\prod_{j=0}^{n-1}
X_{s_{j+1}}^{s_j}(T_{S-s_j}f) \Big]\,,
\end{equation*}
for any measurable bounded function $G$. Take $G$ bounded and  $\mc F_{t_1}$-measurable. Since for any $f\in \mc
S$ the process $X_t^s(f)$ is a martingale, we take the
conditional expectation with respect to
$\mc F_{s_{n-1}}$, and we are lead to 
\begin{equation*}
 \bb E\,\Big[G\,\prod_{j=0}^{n-1}
X_{s_{j+1}}^{s_j}(T_{S-s_j}f) \Big]\;=\;\bb E\,\Big[G\,\prod_{j=0}^{n-2}
X_{s_{j+1}}^{s_j}(T_{S-s_j}f) \Big]\,.
\end{equation*}
By induction, we conclude that
\begin{equation*}
  \bb E\,\Big[G\,\frac{Z_{t_2}}{Z_{t_1}} \Big] \;=\; \bb E\,\Big[G \Big]\,,
\end{equation*}
for any $G$ bounded and $\mc F_{t_1}\!$-measurable, proving that $\{Z_t\,;\,t\geq 0\}$ is, in fact,  a martingale. From
$\bb E\,[Z_{t} | \mc F_s] = Z_{s}$, we get
\begin{equation*}
\begin{split}
 &\bb E\,\Big[\exp\Big\{ \frac{1}{2}\int_0^t\Vert \nabla T_{S-r} f\Vert^2_{L^2(\rho_r)}dr
+i\,\Y_t(T_{S-t} f)\Big\}\Big\vert \mc F_s\Big]\\
&=\; \exp\Bigg\{ \frac{1}{2}\int_0^s\Vert \nabla T_{S-r} f\Vert^2_{L^2(\rho_r)}dr +i\,\Y_s(T_{S-s} f)\Bigg\}\,,
\end{split}
\end{equation*}
which in turn gives
\begin{equation*}
\bb E\,\Big[\exp\Big\{ i\,\Y_t(T_{S-t}
f)\Big\}\big\vert \mc F_s\Big] \;=\; \exp\Bigg\{ -\frac{1}{2}\int_s^t\Vert \nabla T_{S-r} f\Vert^2_{L^2(\rho_r)}dr
+i\,\Y_s(T_{S-s} f)\Bigg\}\,.
\end{equation*}
Note that $T_{S-s} f=T_{t-s} T_{S-t} f$. Thus, writing $g=T_{S-t} f$ we get
\begin{equation*}
\bb E\,\Big[\exp\Big\{ i\,\Y_t(
g)\Big\}\big\vert \mc F_s\Big] \;=\; \exp\Bigg\{ -\frac{1}{2}\int_s^t\Vert \nabla T_{t-r} g\Vert^2_{L^2(\rho_r)}dr
+i\,\Y_s(T_{t-s} g)\Bigg\}\,.
\end{equation*}
Replacing back $g$ by $\lambda f$, where $\lambda\in \bb R$, we obtain
\begin{equation*}
\bb E\,\Big[\exp\Big\{ i\,\lambda \,\Y_t(
f)\Big\}\big\vert \mc F_s\Big] \;=\; \exp\Bigg\{ -\frac{\lambda^2}{2}\int_s^t\Vert \nabla T_{t-r} f\Vert^2_{L^2(\rho_r)}dr
+i\,\lambda\,\Y_s(T_{t-s} f)\Bigg\}\,,
\end{equation*}
which means that, conditionally to $\mc F_s$, the random variable $\Y_t(f)$ has  Gaussian
distribution of mean $\Y_s(T_{t-s}f)$ and variance $\int_s^{t}\Vert \nabla T_{t-r} f\Vert^2_{L^2(\rho_r)}dr$.
Since the distribution at time zero is  determined by \eqref{eq:covar1}, by successively conditioning  we get the uniqueness of the finite dimensional distributions of  the process $\{\Y_t(f)\,;\,t\in{[0,T]}\}$,
 which assures uniqueness in law of the random element $\Y$.
\end{proof}

\subsection{Characterization of limit points} We prove here that any limit point of $\{Q_n\}_{ n\in \mathbb{N}}$  is concentrated on solutions of \eqref{O.U.}, i.e., the limit satisfies (i) and (ii) of Proposition~\ref{pp1}. Fix a test function $f\in \mc S$ (note that $f$ does not depend on time) and let us look at the martingale \eqref{martingaleM} taking $\phi=f$.   Lemma~\ref{lemma32} guarantees the convergence of the martingale $M^n_t(f)$ towards a Brownian motion $W_t(f)$, whose quadratic variation is given by $\int_0^t \|\nabla f\|_{L^2(\rho_r)}^2\,dr$, see the Remark \ref{remark31}. By the hypothesis of the Theorem \ref{thm27}, $\{\mc Y_0^n\}_{n\in\bb N}$ converges, as $n\to\infty$,  to a mean-zero Gaussian field  $\mc Y_0$ with covariance given by \eqref{covar}. Thus $\{\mc Y_0^n(f)\}_{n\in\bb N}$ converges, as $n\to\infty$, to $\mc Y_0(f)$ as well. By tightness proved in Section \ref{s6}, we can pick a subsequence of $\bb N$ such that $\{\mc Y^n_t;t\in [0,T]\}_{n\in \bb N}$  is convergent in the Skorohod topology of $\mc D([0,T],\mc S')$ as $n\to\infty$. Therefore,  $\{\mc Y^n_t(f);t\in [0,T]\}_{n\in \bb N}$ also converges in the Skorohod topology of $\mc D([0,T],\bb R)$, as $n\to\infty$. By abuse of  notation, we denote this subsequence by $n$. Let us look to the integral part of the martingale $M^n_t(f)$. Due to \eqref{int part of mart}, 
\begin{equation*}
\int_0^tn^2\mc L_n \mc Y^n_s(f)\,ds\;=\; \int_0^t \Big\{\mc Y^n_s(\Delta f) + \mc R^n_s(f)\Big\}\,ds\,,
\end{equation*}
where 
\begin{equation*}
\begin{split}
 \mc R^n_s(f)   \;=\; & \mc Y_s^n\Big(\Delta_n f  - \Delta f \Big)  + \sqrt n\,\Big( \nabla_n^+f(0) -f\big(\pfrac{1}{n}\big)\Big)\cdot\big(\eta_{sn^2}(1)-\rho_s^n(1)\big)\\
&+\sqrt n\,\Big( \nabla_n^-f(1)+f\big(\pfrac{n-1}{n}\big)\Big)\cdot \big(\eta_{sn^2}(n-1)-\rho_s^n(n-1)\big)\,.\\
\end{split}
\end{equation*}
Since $f\in \mc S$, it follows that
\begin{equation*}
\lim_{n\to \infty}\mc R^n_s(f)\;=\; 0\,.
\end{equation*}
 On the other hand, by Corollary \ref{cor42} we known that  $\Delta f\in \mc S$, which together with the convergence of $\mc Y^n_t$ gives us that
\begin{equation*}
 \lim_{n\to \infty}\int_0^t \mc Y^n_s(\Delta f)\,ds\;=\;\int_0^t \mc Y_s(\Delta f)\,ds\,,
\end{equation*}
so that
\begin{equation*}
W_t(f)\;=\; \mc Y_t(f) -\mc Y_0(f) -  \int_0^t \mc Y_s(\Delta f)\,ds\,,
\end{equation*}
concluding the characterization of limit points.

\begin{proof}[Proof of Theorem \ref{thm27}] The convergence follows from
 Proposition \ref{pp1}, the previous characterization of limit points and tightness proved in  Section \ref{s6}. It remains only  to prove that
the covariance is as given in  \eqref{covariance non eq limit field}. By \eqref{characterization} we have that 
 \eqref{covariance non eq limit field} is a consequence of  the fact that $W_t$ is Gaussian of variance \eqref{eq212} and that  $\mc Y_0$ and $W_t$ are uncorrelated. This finishes the proof.
\end{proof}

\section{Proof of Theorem \ref{flustat}}\label{s5}
First we make an observation about  the sequence $\rho_{ss}^n(x)$. Consider the  stationary solution of \eqref{hydroeq} denoted by $\overline{\rho}:[0,1]\to[0,1]$. By \cite[Theorem 2.2]{bmns}  $\overline{\rho}(\cdot)$ is given by
$
\overline{\rho}(u)= a\,u\,+\,b$, 
for all $u\in [0,1]$, where
 $a=
\frac{\beta - \alpha}{3}$,  and $  b=
 \alpha +\frac{\beta-\alpha}{3}$.
By \cite[Lemma 3.1]{bmns}, we have 
$\big|\rho_{ss}^n(x)-\overline{\rho}(\pfrac{x}{n})\big|\leq \frac{C}{n}$, for all  $x\in \Sigma_n$, where $C$ does not depend on $x$. 
As we extended $\rho^n_{ss}(\cdot)$ to $0$ and $n$ as $\rho^n_{ss}(0)=\alpha$ and $\rho^n_{ss}(n)=\beta$, this convergence is not true at $x=0$ and $x=n$. But it is not a problem in this paper, here it is enough to have the convergence in $(0,1)$. Moreover, if we needed  the convergence in the whole interval $[0,1]$, we just had to consider the   extension  of $\rho^n_{ss}{(\cdot)}$ to $0$ and $n$ as $\rho^n_{ss}(0)=b_n$ and $\rho^n(n)_{ss}=a_nn+b_n$.
From the previous considerations the Assumption \ref{assumption1} is trivially satisfied when  $\mu_n$ coincides with the stationary measure $\mu_{ss}$. Moreover,  from    \cite[Lemma 3.2]{bmns} the Assumption \ref{assumption2} is also valid in this case. From Theorem \ref{non_eq_flu} we know that the sequence $\{\mc Y_n\}_{n\in\bb N}$ is tight, all limits points satisfy  \eqref{characterization} and $W_t(f)$ is a mean zero Gaussian variable of variance   given by
\begin{equation*}
\begin{split}
&\int_{0}^t \int_0^12\chi(\bar{\rho}(u))(\nabla T_sf(u))^2\,du\,ds\\
&+\int_{0}^t\,\big[\alpha-(1-2\alpha)\bar{\rho}(0)\big](\nabla T_sf(0))^2+ \big[\beta-(1-2\beta)\bar{\rho}(1)\big](\nabla T_sf(1))^2\; ds.\\
\end{split}
\end{equation*}
By Corollary \ref{cor43} we have that $\lim_{t\to\infty}T_tf=0$ in the $L^2{[0,1]}$ norm. Now we analyze the limit of the variance of $W_t(f)$. For that purpose, we take  $\mathcal{Y}_t$ a solution of \eqref{O.U.}, whose covariance is given by \eqref{covariance non eq limit field}.
We want first to compute the asymptotic behavior of this covariance and we claim that it converges to \eqref{stat convariance}. To prove the claim, we note that  by the polarization identity it is enough to analyze the variance. For this purpose, fix $f\in{\mc S}$ and take $g=f$ and $s=t$ in \eqref{covariance non eq limit field} to have that:
\begin{equation*}
 E[(\mathcal{Y}_t(f))^2]\;=\;\sigma(T_tf,T_tf)+\int_0^t\<\nabla T_{t-r} f, \nabla T_{t-r}f\>_{L^2(\rho_r)}dr\,,
\end{equation*}
where  $T_t$ is given in  Definition \ref{def2}. By Corollary \ref{cor421},   $T_tf$ vanishes as $t\rightarrow{+\infty}$, hence the first term at the right hand side of the previous expression  converges to zero, as $t\to\infty$.

Now we analyze the remaining term  that we denote by $R_t(f)$ which,  by \eqref{norm}, is given by
\begin{align}
&\int_{0}^t\int_0^12\chi(\rho(r,u))(\nabla T_{t-r} f(u))^2\,du \,dr\label{51}\\
+&\int_{0}^t\big(\alpha-(1-2\alpha)\rho(r,0)\big)(\nabla T_{t-r} f(0))^2\,dr\label{52}\\
+&\int_{0}^t \big(\beta-(1-2\beta)\rho(r,1)\big)(\nabla T_{t-r} f(1))^2\,dr.\label{53}
\end{align}
We start by dealing with the first term above. Performing an integration by parts in space we can rewrite \eqref{51}  as:
\begin{equation*}
\begin{split}
&\int_0^t2\chi(\rho(r,u))\nabla T_{t-r}f(u) T_{t-r}f(u)\big|_{u=0}^{u=1}\;dr\\
-&\int_0^t\int_0^12\nabla\Big(\chi(\rho(r,u))\nabla T_{t-r}f(u)\Big)\;T_{t-r}f(u)\;du\; dr\,.
\end{split}
\end{equation*}
Since $T_t $ is the semigroup associated to the Laplacian operator of Definition \ref{def2},   then  $\partial_u T_{t-r}f(0)= T_{t-r}f(0)$ and $\partial_u T_{t-r}f(1)= -T_{t-r}f(1)$. As a consequence,  the first term in last expression is equal to
\begin{equation}\label{54}
-\int_0^t\Big(2\chi(\rho(r,1)) (T_{t-r}f(1))^2 +2\chi(\rho(r,0)) (T_{t-r}f(0))^2\Big)\,dr\,,
\end{equation}
 while the second term is equal to
\begin{equation*}
\begin{split}
&-\int_0^t\int_0^12\nabla \chi(\rho(r,u))\;\nabla T_{t-r}f(u)\; T_{t-r}f(u)\; du\; dr \\
&-\int_0^t \int_0^1  2 \chi(\rho(r,u))\;\Delta T_{t-r}f(u)\;T_{t-r}f(u)\; du\; dr.
\end{split}
\end{equation*}
Since $2f\partial_u f=\partial_u f^2$,  last expression becomes
\begin{equation*}
\begin{split}
&-\int_0^t\!\!\!\int_0^1 \!\!\!\nabla \chi(\rho(r,u))\nabla ( T_{t-r}f(u))^2 du\, dr-\int_0^t\!\!\!\int_0^1 \!\!\! 2 \chi(\rho(r,u))(\Delta T_{t-r}f(u))T_{t-r}f(u)du\, \,dr.
\end{split}
\end{equation*}
On the other hand since $\partial_{r}T_{t-r}f=-\Delta T_{t-r}f$, we can rewrite the last expression as:
\begin{equation*}
\begin{split}
&-\int_0^t\!\!\int_0^1\!\!\!\nabla \chi(\rho(r,u))\nabla (T_{t-r}f(u))^2 du\, dr+\int_0^t\!\!\int_0^1 \!\! \! 2 \chi(\rho(r,u))\partial_{r}T_{t-r}f(u)T_{t-r}h(u)du \,dr,
\end{split}
\end{equation*}
which equals to
\begin{equation*}
\begin{split}
&-\int_0^t\int_0^1\nabla \chi(\rho(r,u))\nabla (T_{t-r}f(u))^2 \,du\, dr+\int_0^t\int_0^1 \chi(\rho(r,u))\partial_r(T_{t-r}f(u))^2\, du\, dr\,.\\
\end{split}
\end{equation*}
Integrating by parts in time the second term above, we write the last expression as  
\begin{equation*}
\begin{split}
&-\int_0^t\int_0^1\nabla \chi(\rho(r,u))\,\nabla ( T_{t-r}f(u))^2 \,du\, dr\\
&+\int_0^1 \chi(\rho(t,u))(f(u))^2du-\int_0^1 \chi(\rho(0,u))(T_tf(u))^2du \\
&-\int_0^t\int_0^1 \partial_r\chi(\rho(r,u))(T_{t-r}f(u))^2\,du\, dr\,.\\
\end{split}
\end{equation*}
Integrating by parts in space the first term above, then the previous expression is equal to 
\begin{equation*}
\begin{split}
&-\int_0^t\Big[(1-2\rho(r,1))\nabla \rho(r,1)( T_{t-r}f(1))^2-(1-2\rho(r,0))\nabla \rho(r,0))( T_{t-r}f(0))^2\Big]\, dr\\
&+\int_0^t\int_0^1\Delta \chi(\rho(r,u))( T_{t-r}f(u))^2 \;du\; dr\\
&+\int_0^1 \chi(\rho(t,u))(f(u))^2du-\int_0^1 \chi(\rho(0,u))(T_tf(u))^2du \\
&-\int_0^t\int_0^1 \partial_r\chi(\rho(r,u))(T_{t-r}f(u))^2\;du\; dr\,.\\
\end{split}
\end{equation*}
Since $-\partial_r\chi(\rho(r,u))+\Delta \chi(\rho(r,u))=-2(\nabla \rho(r,u))^2$, last expression is equal to
\begin{equation*}
\begin{split}
&-\int_0^t\Big[(1-2\rho(r,1))\nabla \rho(r,1)( T_{t-r}f(1))^2-(1-2\rho(r,0))\nabla \rho(r,0))( T_{t-r}f(0))^2\Big]\, dr\\
&-\int_0^t\int_0^12(\nabla\rho(r,u))^2( T_{t-r}f(u))^2 \,du\, dr\\
&+\int_0^1 \chi(\rho(r,u))(f(u))^2du-\int_0^1 \chi(\rho(0,u))(T_tf(u))^2du\,.
\end{split}
\end{equation*}
In short, we have that $R_t(f)$ is the sum of the expression above, \eqref{54},  \eqref{52} and \eqref{53}. Then, since $\rho(r,u)$ is the solution of the hydrodynamic equation and using the fact that $\partial_u T_{t-r}f(0)= T_{t-r}f(0)$ and $\partial_u T_{t-r}f(1)= -T_{t-r}f(1)$ we can rewrite $R_t(f)$ as
\begin{align}
&-\int_0^t\int_0^12(\nabla\rho(r,u))^2( T_{t-r}f(u))^2 \;du\; dr\label{55}\\
&+\int_0^1 \chi(\rho(t,u))(f(u))^2du-\int_0^1 \chi(\rho(0,u))(T_tf(u))^2\;du\label{56}\\
&+\int_0^t\Big[2\rho(r,1)(2\beta-1)(T_{t-r}f(1))^2+2\rho(r,0)(2\alpha-1)( T_{t-r}f(0))^2\Big]\; dr\,.\label{57}
\end{align}
At this point we take the limit of $R_t(f)$ as 
$t\rightarrow{+\infty}$. 
By Corollaries \ref{cor421}  and \ref{cor45}, \eqref{56}  converges to
\begin{equation}\label{58b}
\int_0^1 \chi(\overline{\rho}(u))(f(u))^2du\,.
\end{equation}
Denote $g(r,u)=(\nabla\rho(r,u))^2$ for short.  By a change of variables in time, \eqref{55} can be rewritten as
\begin{equation}\label{58a}
-\int_0^t\int_0^1g(t-r,u)\cdot 2\,\big( T_{r}f(u)\big)^2 \,du\, dr\,.
\end{equation}
From $(b)$ in Corollary \ref{corinv}, $\lim_{t\to +\infty} g(t-r,u)=\big(\nabla\overline{\rho}(u)\big)^2$ and $\lim_{r\to+\infty}T_rf=0$. Then, some analysis permits to conclude that \eqref{58a} converges to 
\begin{equation}\label{510a}
-\int_0^1\big(\nabla\overline{\rho}(u)\big)^2 f(u)\,(-\Delta)^{-1}f(u) \;du\,.
\end{equation}
It remains to deal with \eqref{57}. Since $\lim_{t\to+\infty} \rho(r,1)=\overline{\rho}(1)$, $\lim_{t\to+\infty} \rho(r,0)=\overline{\rho}(0)$ and 
$\lim_{r\to+\infty}T_rf=0$, similarly to what we have done above,  we deduce  that \eqref{57} converges to
\begin{equation}\label{510}
\int_0^\infty2\overline{\rho}(1)(2\beta-1)(T_{t}f(1))^2\,dr+\int_0^\infty 2\overline{\rho}(0)(2\alpha-1)( T_{t}f(0))^2\, dr\,.
\end{equation}
In conclusion, the limit of $R_t(f)$ as $t\to+\infty$ is the sum of \eqref{58b}, \eqref{510a} and \eqref{510}, that is,
\begin{equation}\label{58}
\begin{split}
&-\int_0^1(\nabla\overline{\rho}(u))^2 f(u)(-\Delta)^{-1}f(u) \;du+\int_0^1 \chi(\overline{\rho}(u))(f(u))^2du\\
&+2\overline{\rho}(1)(2\beta-1)\int_0^\infty(T_{t}f(1))^2\,dr+2\overline{\rho}(0)(2\alpha-1)\int_0^\infty ( T_{t}f(0))^2\, dr\,.\\
\end{split}
\end{equation}
Since $\nabla\overline{\rho}(u)=\frac{\beta-\alpha}{3}$, $\overline{\rho}(0)=\pfrac{\beta+2\alpha}{3}$ and $\overline{\rho}(1)=\pfrac{2\beta+\alpha}{3}$, we have just proved the claim. 
In particular, the variance of $W_t(f)$ converges, as $t\to\infty$, to \eqref{stat convariance}, so that $W_t(f)$ converges in distribution to a mean zero Gaussian random variable with variance given by \eqref{stat convariance}. Collecting the previous results we get that the random variables  $\mc Y_t(f)$ are mean zero Gaussian with covariance given by \eqref{stat convariance}. Since the process is stationary this ends the proof of Theorem \ref{flustat}.

\section{Tightness}\label{s6}
Now we prove that the sequence of processes $\{\mc Y_t^n; t \in [0,T]\}_{n \in \bb N}$ is tight. Recall that we have defined the density fluctuation field on test functions $f\in\mc S$. Since we want to use  Mitoma's criterium   \cite{Mitoma} for tightness,  we need the following property from the space $\mathcal S$.
\begin{proposition}\label{frechet}
 The space $\mc S$ endowed with the semi-norms given in \eqref{semi-norm}
is a Fr\'echet space.
\end{proposition}
\begin{proof}
The definition of a Fr\'echet space can be found, for instance, in \cite{reedsimon}. Since $C^{\infty}([0,1])$   endowed with the semi-norms \eqref{semi-norm}
is a Fr\'echet space, and a closed subspace of a Fr\'echet space is also a Fr\'echet space, it is enough to show that   $\mc S$ is a closed subspace of $C^\infty([0,1])$., which   is a consequence of the fact that uniform convergence implies point-wise convergence.
\end{proof}
As a consequence of Mitoma's
criterium  \cite{Mitoma} and Proposition \ref{frechet}, the proof of tightness of the $\mc S'$ valued processes  $\{\mc Y_t^n; t \in [0,T]\}_{n \in \bb N}$ follows from tightness of the sequence of real-valued processes $\{\mc Y_t^n(f); t \in [0,T]\}_{n \in \bb N}$,
for $f\in{\mc {S}}$.

\begin{proposition}[Mitoma's criterium,  \cite{Mitoma}]
A sequence  of processes $\{x_t;t \in [0,T]\}_{n \in \bb N}$  in $\mc D([0,T],\mc {S}')$ is tight with respect to the
Skorohod topology if, and only if, the sequence $\{x_t(f);t \in [0,T]\}_{n \in \bb N}$ of real-valued processes is tight with
respect to the Skorohod topology of $\mc D([0,T], \bb R)$, for any $f \in \mc {S}$.
\end{proposition}
Now, to show tightness of the real-valued process we use the Aldous' criterium:
\begin{proposition}
 A sequence $\{x_t; t\in [0,T]\}_{n \in \bb N}$ of real-valued processes is tight with respect to the Skorohod topology of $\mc
D([0,T],\bb R)$ if:
\begin{itemize}
\item[i)]
$\displaystyle\lim_{A\rightarrow{+\infty}}\;\limsup_{n\rightarrow{+\infty}}\;\mathbb{P}_{\mu_n}\Big(\sup_{0\leq{t}\leq{T}}|x_{t
} |>A\Big)\;=\;0\,,$

\item[ii)] for any $\varepsilon >0\,,$
 $\displaystyle\lim_{\delta \to 0} \;\limsup_{n \to {+\infty}} \;\sup_{\lambda \leq \delta} \;\sup_{\tau \in \mc T_T}\;
\mathbb{P}_\rho^\beta(|
x_{\tau+\lambda}- x_{\tau}| >\varepsilon)\; =\;0\,,$
\end{itemize}
where $\mc T_T$ is the set of stopping times bounded by $T$.
\end{proposition}
Fix $f\in{\mc S}$. By \eqref{martingaleM}, it is enough to prove tightness of $\{\mc Y_0^n(f)\}_{n \in
\bb N}$, $\{ \int_{0}^t\Gamma_s^n(f)\, ds; t \in [0,T]\}_{n \in \bb N}$, and $\{\mc M_t^n(f
); t \in [0,T]\}_{n \in \bb N}$. 
\subsection{Tightness at the initial time}
To prove that the sequence $\{\mc Y_0^n(f)\}_{n \in
\bb N}$ is tight, it is enough to  observe that 
\begin{equation*}
\begin{split}
\bb E_{\mu_n}\Big[\Big(\mathcal Y_0^n(f)\Big)^2\Big]&\;=\;\frac{1}{n}\sum_{x=1}^{n-1}f^2\Big(\frac xn\Big)\chi(\rho^n_0(x))+\frac 2n\sum_{x<y}f\Big(\frac xn\Big)f\Big(\frac yn\Big)\varphi^n_0(x,y)
\end{split}
\end{equation*}
and by 
Assumption \ref{assumption2} last expression is clearly bounded. 

\subsection{Tightness of the martingales}
By Lemma \ref{lemma32} since the sequence of martingales converges, in particular, it is tight.

\subsection{Tightness of the integral terms}

The first claim of Aldous' criterium, can be easily checked for the integral term $\int_{0}^t\Gamma_s^n(f)\, ds$.
Since $f\in{\mc{S}}$  and by the Cauchy-Schwarz inequality we have that
\begin{equation*}
\mathbb{E}_{\mu_n}\Big[\sup_{t\leq {T}}\Big(\int_{0}^t\Gamma_s^n(f)\, ds\Big)^2\Big]\leq T \int_{0}^T \mathbb{E}_{\mu_n}\Big[\Big(\frac{1}{\sqrt{n}}\sum_{x=1}^{n-1}\Delta_n f(\tfrac{x}{n})(\eta_{sn^2}(x)-\rho^n_s(x))\Big)^2\Big]\, ds
\end{equation*}
plus a term $O\Big(\frac{1}{n}\Big)$.
The term on the right hand side of last expression is bounded from above by $T^2$ times
\begin{equation}\label{estimate}
\frac{1}{{n}}\sum_{x=1}^{n-1}\big(\Delta_n f(\tfrac{x}{n})\Big)^2\sup_{t\leq{T}}\chi(\rho^n_t(x))
+\frac{1}{{n}}\sum_{\at{x\neq y}{x,y=1}}^{n-1}\Delta_n f(\tfrac{x}{n})\Delta_n f(\tfrac{y}{n})\sup_{t\leq{T}}\varphi^n_t(x,y)\,,
\end{equation}
where $\varphi^n_t(x,y)$ is given in \eqref{cov}.
Now, by Proposition \ref{prop3.1} and since $f\in\mc S$ last expression is bounded by a constant. Now we need to check the second claim. For that purpose,
fix a stopping time $\tau \in \mc T_T$. By the Chebychev's inequality together  with \eqref{estimate}, we get that
 \begin{equation*}
\mathbb{P}_{\mu_n}\Big(\Big|  \int_{\tau}^{\tau+\lambda}\!\!\Gamma^n_s(f)\, ds\;\Big| >\varepsilon\Big)
	\leq \frac{1}{\varepsilon^2} \mathbb{E}_{\mu_n}\Big[ \Big(  \int_{\tau}^{\tau+\lambda}\!\!\Gamma^n_s(f)\; ds \;\Big)^2\Big]
	\leq \frac{\delta^2 C}{\varepsilon^2}\,,
\end{equation*}
which vanishes as $\delta\rightarrow{0}$.

\section{Discrete equations}\label{s7}

In this section we prove some technical estimates that are needed along the paper.  
\subsection{Two point correlation function}

\begin{definition}[Two-point correlation function]

 For each $x,y\in\Sigma_n$, $x<y$, and $t\in[0,T]$, we define  the two-point correlation function as
 \begin{equation}\label{cov}
\varphi^n_t(x,y)\;=\;\mathbb{E}_{\mu_n}[\eta_{tn^2}(x)\eta_{tn^2}(y)]-\rho^n_t(x)\rho^n_t(y)\,,
\end{equation}
where  $\rho^n_t$ was defined in \eqref{rho_t}. 
Moreover, for $x=0$ or $y=n$, we set  $\varphi_t^n(x,y)=0,$ 
\end{definition}
\begin{proposition}\label{prop3.1}
There exists $C>0$ such that
\begin{equation}\label{eq181}
\sup_{t\geq 0}\max_{(x,y)\in V_n}|\varphi_t^n(x,y)|\;\leq \; \frac{C}{n}\,,
\end{equation}
 where
$V_n\;=\;\{(x,y)\,;\, x,y\in \bb N ,\, 0< x<y<   n\}$.
\end{proposition}
\begin{proof}
First, observe that $\varphi^n_t(x,y)$ can be rewritten as $$\bb E_{\mu_n}[(\eta_{tn^2}(x)-\rho^n_t(x))(\eta_{tn^2}(y)-\rho^n_t(y))]\,,$$ so that from  Kolmogorov's forward equation, we have  that 
\begin{equation*}
\p_t \varphi^n_t(x,y) \;=\;\bb E_{\mu_n}\Big[(n^2\mc L_n+\p_t) (\eta_{tn^2}(x)-\rho^n_t(x))(\eta_{tn^2}(y)-\rho^n_t(y))\Big]\,.
\end{equation*}
Applying \eqref{lnb} and \eqref{lno}  and performing 
some long, but elementary, calculations we deduce that $\varphi_t^n$ solves the following system of ODE's:
\begin{equation}\label{ODEsystem_phi}
\begin{cases}
\p_t \varphi_t^n(x,y)=n^2 \A_n \varphi^n_t(x,y) +g_t^n(x,y)\,, & \textrm{ for } (x,y)\in V_n\,,\; t>0\,,\\
\varphi_t^n(x,y)=0\,, & \textrm{ for } (x,y)\in \p V_n\,, \;t>0\,,\\
\varphi_0^n(x,y)=\bb E_{\mu_n}[\eta_0(x)\eta_0(y)]-\rho_0^n(x)\rho_0^n(y)\,, & \textrm{ for } (x,y)\in V_n\cup \p V_n\,,\\
\end{cases}
\end{equation}
where $\A_n$ is the linear operator that acts on functions $f:V_n\cup \p V_n \to\bb R$ as
 \begin{equation*}
 (\A_nf)(u)=\sum_{v\in V_n} c_n(u,v)\big[ f(v)-f(u)\big]\,,\quad \textrm{for} \; u\in{V_n}\,,
 \end{equation*}
 with
 \begin{equation*}
c_n(u,v)\;=\; \begin{cases}
 1\,, & \textrm{ if } \; \Vert u-v\Vert =1 \textrm{ and } \;u, v\in  V_n\,,\\
 n^{-1}\,, & \textrm{ if }\;  \Vert u-v\Vert=1\textrm{ and } u\in V_n,\; v\in \p V_n\,,\\ 
 0\,,& \textrm{ otherwise, }
 \end{cases}
 \end{equation*} and
 $\p V_n$ represents the boundary of the set $V_n$, which we define as  $$\p V_n\;=\; \{(0,1),\ldots, (0,n)\}\cup \{(1,n),\ldots, (n-1,n)\}\,,$$
see Figure \ref{fig2} for an illustration. 
Above we have that 
 \begin{equation}\label{g}
 g_t^n(x,y)\;=\;n^2(\rho^n_t(x)-\rho^n_t(x+1))^2\,\cdot\,\Ind{\mc D_n}(x,y)\,,
 \end{equation}
 where the diagonal $\mc D_n$ is defined by
  \begin{equation*}
  \mc D_n\,=\,\{(x,y)\in V_n;\, y=x+1\}
  \end{equation*}
  and $\Ind{\Gamma}$ is the indicator function of the set $\Gamma$.

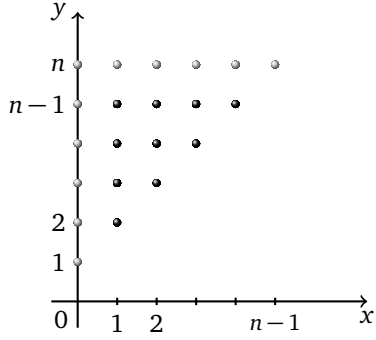
\begin{figure}[!htb]
\centering
\begin{tikzpicture}[thick, scale=0.7]
\draw[->] (-0.5,0)--(5.5,0) node[anchor=north]{$x$};
\draw[->] (0,-0.5)--(0,5.5) node[anchor=east]{$y$};
\begin{scope}[scale=0.75]
\foreach \x in {1,...,4} 
	\foreach \y in {\x,...,4}
		\shade[ball color=black](\x,1+\y) circle (0.1);
		
\foreach \x in {0,...,5} 
	\shade[ball color=gray!50](\x,6) circle (0.1); 
	 
\foreach \y in {1,...,5} 
	\shade[ball color=gray!50](0,\y) circle (0.1);  
\end{scope}	
\draw (0,0) node[anchor=north east] {$0$};
\draw (0.75,2pt)--(0.75,-2pt) node[anchor=north] {$1$};
\draw (1.5,2pt)--(1.5,-2pt) node[anchor=north]{$2$};
\draw (2.25,2pt)--(2.25,-2pt);
\draw (3,2pt)--(3,-2pt);
\draw (3.75,2pt)--(3.75,-2pt) node[anchor=north]{\small $n-1$};
\draw (-0.05,.75) node[anchor=east] {$1$};
\draw (-0.05,1.5) node[anchor=east]{$2$};
\draw (-0.05,3.75) node[anchor= east]{$n-1$};
\draw (-0.05,4.5) node[anchor=east]{$n$};
\end{tikzpicture}
\caption{Black balls are elements of $V_n$ and gray balls are elements of $\p V_n$.}\label{fig2}
\end{figure}

Above,  $\Vert \cdot \Vert$ denotes the supremum norm.
Note that $\A_n$ is the generator of a random walk in $V_n\cup \p V_n$, , denoted by $\{X_{tn^2}; \,t\geq 0\}$, which has  jump rates given by $c_n(u,v)$ and  is absorbed at the boundary $\p V_n$. Denote by $\mathbf{P}_{u}$ and $\mathbf{E}_{u}$   the corresponding probability and expectation, respectively, starting from the position $u\in V_n$.
Now we introduce the function
 \begin{equation}\label{eq17}
\phi_t^n(x,y)\;=\; {\bf E}_{(x,y)}\Big[\varphi^n_0( X_{tn^2})+ \int_0^t g^n_{t-s}( X_{sn^2})\,ds\Big]\,,
\end{equation}  
with $\varphi^n_0$ and $g^n_t$ given above. 
Since ${\bf E}_{(x,y)}[f( X_{tn^2})]=(e^{tn^2\A_n}f)(x,y)$ is a semigroup and by using Kolmogorov's forward equation and Leibniz Integral Rule, we can show that  the function $\phi_t^n$ is solution of the semi-linear system \eqref{ODEsystem_phi}, so that $\phi_t^n=\varphi_t^n$. 
Therefore, in order to prove the proposition we just have to estimate the two terms at the right hand side of  last display, since 
\begin{equation*}\label{eq17a}
\max_{(x,y)\in V_n}|\varphi_t^n(x,y)|\;\leq\;\max_{(x,y)\in V_n}|\varphi^n_0(x,y)|+ \max_{(x,y)\in V_n}\Big|{\bf E}_{(x,y)}\Big[\int_0^t g^n_{t-s}( X_{sn^2})\,ds\Big]\Big|\,.
\end{equation*}
From Assumption \ref{assumption2},  the first term on the right hand side of last expression is bounded from above by $c/n$.
It remains to deal with the second term. Note that since
the operator $n^2\A_n$ is a bounded operator (for $n$ fixed) it generates  a uniformly continuous semigroup $\{e^{sn^2\A_n};\,s\geq 0\}$ on $ V_n\cup\p V_n$.
By Fubini's Theorem
\begin{equation*}\label{eq20a}
{\bf E}_{(x,y)}\Big[\int_0^t g^n_{t-s}( X_{sn^2})\,ds\Big]= \int_0^t \big(e^{sn^2\A_n}g^n_{t-s}\big)(x,y)\,ds\,.
\end{equation*}
Changing variables, the right hand side of last expression can be written as 
$$\int_0^t \big(e^{(t-r)n^2\A_n}g^n_r\big)(x,y)\,dr \,.$$
Thus, the proof ends as a consequence of  the next lemma.

\end{proof}

Before stating the next lemma, we notice that for $u,v\in V_n\cup\p V_n$
\begin{equation}\label{eq18}
e^{tn^2\A_n}(u,v)={\bf P}_{u}\Big[X_{tn^2}=v\Big]\,.
\end{equation}

\begin{lemma}
There exists $C>0$ which does not depend on $n$ such that
\begin{equation*}
\sup_{t\geq 0}\max_{(x,y)\in V_n}\Big|\int_0^t \big(e^{(t-r)n^2\A_n}g^n_r\big)(x,y)\,dr\Big|\;\leq \; \frac{C}{n}\,.
\end{equation*}
\end{lemma}
\begin{proof}
Since the function $g^n_r$ defined in \eqref{g} is supported on the diagonal $\mc D_n$, we can rewrite  $(e^{(t-r)n^2\A_n}g^n_r)(x,y)$ as 
\begin{equation*}
\sum_{z=1}^{n-2}e^{(t-r)n^2\A_n}\big((x,y),\,(z,z+1)\big)\,g^n_r(z,z+1)\,.
\end{equation*}
Then, for all $(x,y)\in V_n$,
\begin{equation}\label{eq22}
\Big|\int_0^t \big(e^{(t-r)n^2\A_n}g^n_r\big)(x,y)\,dr\Big|\;\leq \; S_n\cdot\int_0^t \sum_{z=1}^{n-2}e^{(t-r)n^2\A_n}\big((x,y),\,(z,z+1)\big)\,dr\,,
\end{equation}
where 
\begin{equation}\label{S_n}
S_n\;=\;\sup_{ r\geq 0}\max_{z\in\{1,\dots,n-2\}}|g^n_r(z,z+1)|\,.
\end{equation}
 First we will estimate the time integral at the right hand side of \eqref{eq22} and then we will estimate $S_n$.
  By  \eqref{eq18} together with  a change of   variables and by the  definition of  $\mc D_n$, we get
\begin{equation*}
\int_0^{tn^2}\sum_{z=1}^{n-2}{\bf P}_{(x,y)}\big[X_{s}=(z,z+1)\big]\,\frac{ds}{n^2}
\;=\;
\int_0^{tn^2}{\bf P}_{(x,y)}\big[X_{s}\in \mc D_n\big]\,\frac{ds}{n^2}\,.
\end{equation*}
 Extending the interval of integration to infinity and applying Fubini's Theorem on  the last integral, we bound it from above by 
\begin{equation*}
\frac{1}{n^2}\,{\bf E}_{(x,y)}\Big[\int_0^{\infty}\Ind{X_s \in \mc D_n}\,ds\Big]\,.
\end{equation*}
Notice that the expectation above is  the total time spent by the random walk $\{ X_s;\,s\geq 0\}$ on the diagonal $\mc D_n$. By Section 3 of \cite{bmns}, we have the following bound
\begin{equation*}
{\bf E}_{(x,y)}\Big[\int_0^{\infty}\Ind{X_s \in \mc D_n}\,ds\Big]\;\leq\;C\,n\,,
\end{equation*}
 for all $(x,y)\in V_n$.
Thus  the integral at the right hand side of \eqref{eq22} is bounded from above by $C/n$. In order  to conclude the proof, we need to prove that  $S_n$, which was defined in \eqref{S_n},  is bounded. By the definition of $g_r^n$ given in \eqref{g}, it is enough to prove that 
\begin{equation*}
\big|\rho^n_t(x+1)-\rho^n_t(x)\big|\;\leq\;\frac{C}{n}\,,
\end{equation*}
for all $x\in\{1,\dots,n-2\}$ and uniformly in $t\geq 0$ and this  follows from Proposition \ref{disc_cont_prox}.
\end{proof}
\subsection{Estimates for the discrete equation}
 \begin{proposition} \label{disc_cont_prox}Let $\rho^n_t(\cdot)$ be the solution of  \eqref{disc_heat}. Then, there exists $C>0$ which does not depend on $n$ such that
 \begin{equation}\label{eq23}
\big|\rho^n_t(x+1)-\rho^n_t(x)\big|\;\leq\;\frac{C}{n}\,,
\end{equation}
for all $x\in\{1,\dots,n-2\}$,  uniformly in $t\geq 0$.
 \end{proposition}
 \begin{proof} 
 Let $\rho(t,u)$ be the solution of the equation
 \begin{equation}\label{eq_hydro_desc}
\begin{cases}
\p_t \rho(t,u)\;=\; \p_u^2 \rho(t,u)\,, & \textrm{ for } t>0\,,\, u\in (0,1)\,,\\
\p_u \rho(t,0) \;=\; \rho(t,0^+)-\rho(t,0)\,, & \textrm{ for } t>0\,,\\
\p_u \rho(t,1) \;=\; \rho(t,1)-\rho(t,1^-)\,, & \textrm{ for } t>0\,,\\
\rho(t,0)\;=\;\alpha\,,\;\; \rho(t,1)=\beta\,,& \textrm{ for } t>0\,,\\
\rho(0,u)\;=\;\rho_0(u)\,,& u\in [0,1]\,.
\end{cases}
\end{equation}
We notice that $\rho_t$ is essentially the solution of the hydrodynamic equation given in \eqref{hydroeq},  but discontinuous at $0$ and $1$. By the Corollary \ref{cor43}, we have assured the smoothness of $\rho(t,u)$ in $(0,1)$.

Let  $\gamma^n_t(x):=\rho^n_t(x)-\rho_{t}(\tfrac xn)\,$  for $x\in{\Sigma_n}\cup\{0,n\}$.  Then $\gamma^n_t$ satisfies the equation
\begin{equation}\label{eq28}
\left\{
\begin{array}{ll}
 \partial_t\gamma_t^n(x)\,=\,(n^2\mc B_n \gamma_t^n)(x)+F_t^n(x)\,, \;\; x\in\Sigma_n\,,\;\;t \geq 0\,,\\
 \gamma_t^n(0)=0\,, \quad \gamma^n_t(n)=0\,, \;\;t \geq 0\,,\\
\end{array}
\right.
\end{equation}
where, for $x\in\{2,\ldots,n-2\}$, $F_t^n$ accounts for the difference between discrete and continuous Laplacians and, for $x
\in\{1,n-1\}$,  $F_t^n(x)=(n^2\mathcal{B}_n-\partial_u^2)\rho_t(\tfrac xn)$.

In order to prove \eqref{eq23}, we add and subtract $\rho_{t}\big(\frac{x+1}{n}\big)$ and $\rho_{t}\big(\frac{x}{n}\big)$ to $|\rho^n_t(x+1)-\rho^n_t(x)|$ and  use the triangle inequality to have that 
\begin{equation*}
\big|\rho^n_t(x+1)-\rho^n_t(x)\big|\;\leq\;\big|\gamma^n_t(x+1)\big|+\big|\gamma^n_t(x)\big|+\Big|\rho_{t}(\tfrac{x+1}{n})-\rho_{t}(\tfrac{x}{n})\Big|\,.
\end{equation*}
Since $\rho_t$ is smooth in $(0,1)$, it remains to show that $\gamma^n_t$ is bounded by $c/n$.  For that purpose, 
let $\{\XX_s,\, s\geq 0\}$ be the random walk on ${\Sigma_n}\cup\{0,n\}$, with generator $\mc B_n$, absorbed at the boundaries $\{0,n\}$. Denote by $E_x$ the  expectation with respect to the probability induced by the generator $\mc B_n$ and the initial position $x$.  As before,  we can write the solution of  \eqref{eq28}
  as 
\begin{equation*}
\gamma_t^n(x)\;=\; E_x\Big[\gamma^n_0(\XX_{tn^2})+\int_{0}^t F_{t-s}^n(\XX_{sn^2})\,ds\Big]\,.
\end{equation*}
Then,
\begin{equation*}
\sup_{t\geq 0}\max_{x\in\Sigma_n}|\gamma_t^n(z)|\;\leq\; \max_{x\in \Sigma_n}|\gamma^n_0(x)|\;+\;\sup_{t\geq 0}\max_{x\in \Sigma_n}\Big|E_x\Big[\int_{0}^t F_{t-s}^n(\XX_{sn^2})\,ds\Big]\Big|\,.
\end{equation*}
Since $\gamma^n_0(x)=|\rho^n_0(x)-\rho_0(x)|$, by Assumption \ref{assumption1} we only need to control the second  term at the right hand side of the previous expression.
Repeating the same strategy as before, we decompose the expectation above into the possible positions of the chain at time $s$ and we are left to estimate 
\begin{equation}\label{eq:imp}
\int_{0}^{t}\sum_{z=1}^{n-1}P_x\Big[\XX_{sn^2}=z\Big]\cdot F_{t-s}^n(z)\,ds\,.
\end{equation}
Since the discrete Laplacian approximates the continuous Laplacian, we conclude  that  $F_t^n(x)\le{C/n^2}$ for any $x\in\{2,\ldots,n-2\}$ and for any $t\geq 0$. Therefore, we can bound \eqref{eq:imp} by
\begin{equation}\label{713}
\frac{C}{n^2}+\sum_{k\in \{1,n-1\}} E_x\Big[\int_{0}^{\infty}\textbf{1}_{\{\XX_{sn^2}=k\}}\, ds\Big]\cdot|F_t^n(k)|\,.
\end{equation}
 Moreover,  we also have that
\begin{equation*}
\begin{split}
F_t^n(1)\;=\;&n^2\,\Big(\rho_t\big(\pfrac{2}{n}\big)-\rho_t\Big(\pfrac{1}{n}\Big)\Big)-n\,\Big(\rho_t\big(\pfrac{0}{n}\big)-\rho_t\big(\pfrac{1}{n}\big)\Big)-\partial_u\rho_t\big(\pfrac{1}{n}\big)\\
\;=\;&n\,\Big(\partial_u\rho_t\big(\pfrac{1}{n}\big)-\rho_t\big(\pfrac{0}{n}\big)-\rho_t\big(\pfrac{1}{n}\big)\Big)+O(1)\,,
\end{split}
\end{equation*}
and by the boundary conditions in \eqref{eq_hydro_desc} we obtain that $|F_t^n(1)|\leq  C$ for any $t\geq 0$. For $k=n-1$ we obtain exactly the same bound as for $k=1$. 

The  expectation in \eqref{713} is  the average time spent by the random walk at the site $k$ until its absorption. As an application of the Markov Property, it can be expressed as the solution of the elliptic equation  
\begin{equation*}
\left\{
\begin{array}{ll}
-\mc B_n\psi^n(x)\;=\;C\delta_{x=k}\,, \;\;\forall\; x\in\Sigma_n\,,\\
 \psi^n(0)\;=\;0\,, \quad \psi^n(n)\;=\;0\,,\\
\end{array}
\right.
\end{equation*}
 where $C$ is a constant.
 A simple computation shows that, for $k=1$, 
 \begin{equation*}
 \psi^n(x)\;=\;-\frac{1}{3n^2-2n}x+\frac{2n-1}{3n^2-2n}\,,\qquad \forall \;x\in\Sigma_n\,, 
\end{equation*}
so that $\max_{x=1,\ldots,n-1}|\psi^n(x)|\leq  C/n$. For $k=n-1$ the same bound holds. Putting all the estimates together,  the proof ends. 
\end{proof}

\section*{Acknowledgements}
T. F. is supported by FAPESB through the project Jovem Cientista-9922/2015. P. G. thanks  FCT/Portugal for support through the project 
UID/MAT/04459/2013. A. N. thanks FAPERGS and L'OREAL for support through the projects  002063-2551/13-0 and ``L'OR\' EAL - ABC - UNESCO Para Mulheres na Ci\^encia'', respectively.

\bibliographystyle{plain}
\bibliography{bibliography}

\end{document}